\documentclass[12pt]{amsart}  	% use "amsart" instead of "article" for AMSLaTeX format

\usepackage[margin=1in]{geometry}  % set the margins to 1in on all sides

\usepackage[all]{xy}						% TeX will automatically convert eps --> pdf in pdflatex	
\usepackage{amssymb}
\usepackage{amscd,latexsym,amsthm,amsfonts,amssymb,amsmath,amsxtra}
\usepackage[mathscr]{eucal}
\usepackage[colorlinks]{hyperref}
\usepackage{mathtools}
\usepackage{tikz-cd}

\usepackage{chngcntr}
\numberwithin{equation}{section}
\newcounter{keepeqno}

\hypersetup{linkcolor=black, citecolor=black}

\newcommand{\BA}{{\mathbb {A}}}

\newcommand{\BC}{{\mathbb {C}}}

\newcommand{\BG}{{\mathbb {G}}}

\newcommand{\BL}{{\mathbb {L}}}

\newcommand{\CC}{{\mathcal {C}}}

\newcommand{\CM}{{\mathcal {M}}}

\newcommand{\CS}{{\mathcal {S}}}

\newcommand{\Fd}{{\mathfrak {d}}}
\newcommand{\Fe}{{\mathfrak {e}}}

\newcommand{\Fg}{{\mathfrak {g}}}

\newcommand{\Fn}{{\mathfrak {n}}}

\newcommand{\Fp}{{\mathfrak {p}}}

\newcommand{\Fs}{{\mathfrak {s}}}

\newcommand{\RI}{{\mathrm {I}}}

\newcommand{\ScF}{{\mathscr {F}}}

\newcommand{\ScI}{{\mathscr {I}}}

\newcommand{\ScL}{{\mathscr {L}}}

\newcommand{\ScZ}{{\mathscr {Z}}}

\newcommand{\Ad}{{\mathrm{Ad}}}
\newcommand{\ad}{{\mathrm{ad}}}

\newcommand{\bkn}{{\mathrm{BKN}}}

\newcommand{\der}{{\mathrm{der}}}

\newcommand{\End}{{\mathrm{End}}}

\newcommand{\fp}{{\mathrm{fp}}}

\newcommand{\GL}{{\mathrm{GL}}}

\newcommand{\I}{{\mathrm{I}}}
\newcommand{\Id}{{\mathrm{Id}}}

\newcommand{\pr}{{\mathrm{pr}}}

\newcommand{\PGL}{{\mathrm{PGL}}}

\renewcommand{\Re}{{\mathrm{Re}}}

\newcommand{\rsc}{{\mathrm{sc}}}

\newcommand{\SL}{{\mathrm{SL}}}

\newcommand{\Spec}{{\mathrm{Spec}}}
\newcommand{\SO}{{\mathrm{SO}}}

\newcommand{\Sp}{{\mathrm{Sp}}}

\newcommand{\st}{{\mathrm{st}}}

\newcommand{\ud}{\,\mathrm{d}}

\newcommand{\wh}{\widehat}

\newcommand{\ol}{\overline}

\def\alp{{\alpha}}

\def\diag{{\rm diag}}

\def\vsig{{\varsigma}}
\def\sym{{\rm sym}}
\def\sig{{\sigma}}

\def\ome{{\omega}}

\def\lam{{\lambda}}

\def\gam{{\gamma}}

\newcommand{\sslash}{\mathbin{/\mkern-6mu/}}

\newtheorem{thm}{Theorem}[section]
\newtheorem{dfn}[thm]{Definition}

\newtheorem{prp}[thm]{Proposition}
\newtheorem{lem}[thm]{Lemma}
\newtheorem{cor}[thm]{Corollary}

\newtheorem{cnj}[thm]{Conjecture}

\makeatletter

\newcommand{\Rmnum}[1]{\expandafter\@slowromancap\romannumeral #1@}
\makeatother

\begin{document}

\title[On the Braverman-Kazhdan-Ng\^o Triples]
{On the Braverman-Kazhdan-Ng\^o Triples}

\author{Dihua Jiang, Zhaolin Li, and Guodong Xi}	
\address{School of Mathematics, University of Minnesota, 206 Church St. S.E., Minneapolis, MN 55455, USA.}
\email{jiang034@umn.edu}
\email{li001870@umn.edu}
\email{xi000023@umn.edu}


\date{\today}

\thanks{The research of this paper is supported in part by the NSF Grant DMS-2200890.}

\begin{abstract}
In the Braverman-Kazhdan proposal (\cite{BK00}) and certain refinement of Ng\^o (\cite{Ngo20}) for automorphic $L$-functions, the reductive group $G$ and the representations $\rho$ of the Langlands dual group $G^\vee$ are taken with certain assumptions. We introduce the notion of the Braverman-Kazhdan-Ng\^o triples 
$(G,G^\vee,\rho)$ and show that for general automorphic $L$-functions, it is enough to consider the Braverman-Kazhdan-Ng\^o triples (Theorem \ref{thm:localm}). We also verify that for a given Braverman-Kazhdan-Ng\^o triple, 
the reductive monoid constructed from the Vinberg method and that constructed from the Putcha-Renner method are isomorphic (Theorem \ref{thm:monoids}). 
\end{abstract}

\maketitle
\tableofcontents
   
%%%%%%%%%%%%%%%%%%%%%%%%
\section{Introduction}
%%%%%%%%%%%%%%%%%%%%%%%

Let $G$ be a split reductive algebraic group defined over a number field $k$ and $G^\vee$ be the Langlands dual group of $G$. For an irreducible finite-dimensional representation $(\rho,V_\rho)$ of $G^\vee(\BC)$ and an irreducible cuspidal automorphic representation $\pi$ of $G(\BA)$, where $\BA$ is the ring of adeles of $k$, R. Langlands introduced in \cite{L70} the automorphic $L$-function 
$L(s,\pi,\rho)$ associated with $(\pi,\rho)$ and proved that $L(s,\pi,\rho)$ can be defined as an Euler product of local $L$-factors 
\[
L(s,\pi,\rho)=\prod_{\nu\in|k|}L_\nu(s,\pi_\nu,\rho)
\]
when $s\in\BC$ with $\Re(s)$ sufficiently positive, where $|k|$ is the set of all local places of $k$, $\pi=\otimes_\nu\pi_\nu$ (the restricted tensor product decomposition), and 
the local $L$-factors $L_\nu(s,\pi_\nu,\rho)$ can be defined through the local Langlands conjecture. It is well known that the local $L$-factors are well-defined when $\nu$ is archimedean or any finite local place at which $\pi_\nu$ is unramified. The Langlands conjecture predicts that $L(s,\pi,\rho)$ admits a meromorphic continuation to 
$s\in\BC$ and satisfies the standard global functional equation. One of the central problems in the theory of automorphic forms and in the Langlands program is to prove the Langlands conjecture for general automorphic $L$-functions $L(s,\pi,\rho)$. 

The Langlands-Shahidi method and the Rankin-Selberg method established the Langlands conjecture for long lists of cases. However, the general situation of the Langlands conjecture is still widely open. Around the year 2000, A. Braverman and D. Kazhdan (\cite{BK00}) proposed a framework to establish the Langlands conjecture for general $L$-functions $L(s,\pi,\rho)$ via local and global harmonic analysis on $G$, which generalizes 
the method of J. Tate's thesis and the work of R. Godement and H. Jacquet to the great generality. In the Braverman-Kazhdan proposal \cite{BK00} and in its refinement by B. C. Ng\^o in \cite{Ngo20}, the triples $(G,G^\vee,\rho)$ 
should satisfy the following properties.

\begin{dfn}[Braverman-Kazhdan-Ngo Triple]\label{dfn:BKN}
Let $k$ be a field of characteristic zero. Let $G$ be a reductive $k$-split algebraic group, $G^\vee$ be the Langlands dual group, which is also a reductive $k$-split algebraic group, and $\rho$ be an irreducible finite-dimensional $k$-rational representation of $G^\vee$. 
A given triple $(G, G^\vee,\rho)$ is called a {\bf Braverman-Kazhdan-Ngo triple}, or simply, the {\bf BKN-triple} if it satisfies the following conditions. 
\begin{enumerate}
    \item The kernel of $\rho\colon G^\vee\to\GL_{d_\rho}$ is trivial, where $d_\rho$ is the dimension of $\rho$.
    \item There exists a character $\Fd\colon G\to\BG_m$, such that 
the following exact sequence
\[
1\longrightarrow G_0 \longrightarrow G \stackrel{\Fd}{\longrightarrow} \BG_m\longrightarrow 1
\]
holds with $G_0=[G,G]$ the derived group of $G$. 
\item The derived group $G_0=[G,G]$ is simply connected. 
\item The dual morphism $\Fd^\vee\colon \BG_m\to G^\vee$ composed with the given representation $\rho$ is a scalar multiplication of $\BG_m$ on the vector space $V_\rho$, i.e., the identity 
$\rho(\Fd^\vee(a))=a\cdot\RI_{V_\rho}$
holds for any $a\in\BG_m(\BC)=\BC^\times$.
\end{enumerate}
The {\rm BKN}-triple is denoted by 
\begin{align}\label{BKT}
    \Delta_{\bkn}(\Fd):=(G,G^\vee,\rho)_\Fd=(G_\Fd,G_\Fd^\vee,\rho^\Fd).
\end{align}
\end{dfn}

Note that although the Braverman-Kazhdan proposal can be formulated for general reductive groups defined over a number 
filed $k$, for discussion in this paper, we assume that $G$ is $k$-split. 

The local conjecture in the Braverman-Kazhdan proposal for any local filed $F$ of characteristic zero can be stated as follows. Some more detailed comments on this conjecture can be found in \cite[Section 4]{Ngo20} and \cite[Section 2.1]{LN24}.

\begin{cnj}[Braverman-Kazhdan]\label{cnj:LBK}
     Let $F$ be a local field of characteristic zero and $\Delta_{\bkn}(\Fd)=(G,G^\vee,\rho)_\Fd$ be a {\rm BKN}-triple with $G$ being $F$-split. Let $\pi$ be an irreducible admissible representation of $G(F)$, which is of the Casselman-Wallach type if $F$ is archimedean. 
     Then the following statements hold. 
\begin{enumerate}
    \item {\bf Local Zeta Integral:} There exists a Schwartz space $\CS_\rho(G(F))$ with 
    \[
    \CC_c^\infty(G(F))\subset \CS_\rho(G(F))\subset \CC^\infty(G(F))
    \]
    such that the local zeta integral
    \[
    \ScZ(s,\varphi_\pi,\phi)=\int_{G(F)}\varphi_\pi(g)\phi(g)|\Fd(g)|_F^{s+\frac{n_\rho-1}{2}}\ud g,\quad 
    {\rm for}\ \varphi_\pi\in\CC(\pi)\ {\rm and}\ \phi\in\CS_\rho(G(F))
    \]
    converges absolutely for $\Re(s)$ sufficiently positive and admits a meromorphic continuation to $s\in\BC$, where 
    $n_\rho$ is the normalization as given in \cite{Ngo20}, $\CC(\pi)$ is the space of matrix coefficients of $\pi$, 
    $\CC^\infty(G(F))$ is the space of smooth functions on $G(F)$ and $\CC_c^\infty(G(F))$ is its subspace of compactly supported functions.
    \item {\bf Local $L$-Factor:} The local zeta integrals $\ScZ(s,\varphi_\pi,\phi)$ are holomorphic mulfiples of 
    the local Langlands $L$-factor $L(s,\pi,\rho)$. In the $p$-adic case, the space
    \[
    \ScI(\pi)=\{\ScZ(s,\varphi_\pi,\phi)\mid \varphi_\pi\in\CC(\pi)\ {\rm and}\ \phi\in\CS_\rho(G(F))\}
    \]
    is a non-zero fractional ideal of $\BC[q^s,q^{-s}]$ generated by $L(s,\pi,\rho)$, 
    where $q$ is the cadinarity of the residue field of $F$. In the archimedean case, for any vertical strip 
    \[
    S_{a,b}=\{s\in\BC\mid a<\Re(s)<b\}
    \]
    with $a<b$, if $p(s)\in\BC[s]$ is such polynomials that $p(s)L(s,\pi,\rho)$ is bounded on $S_{a,b}$,  
    with small neighborhoods at the possible poles of the $L$-function $L(s,\pi,\rho)$ removed,
    then $p(s)\ScZ(s,\varphi_\pi,\phi)$ are also bounded on $S_{a,b}$, with small neighborhoods at the possible poles of the $L$-function $L(s,\pi,\rho)$ removed, for any $\varphi_\pi\in\CC(\pi)$ and $\phi\in\CS_\rho(G(F))$. 
    \item {\bf Basic Function:} There is a basic function $\BL_\rho\in\CS_\rho(G(F))$ such that when $\pi$ is unramified with a normalized zonal spherical function $\varphi_\pi^\circ\in\CC(\pi)$, the identity 
    \[
    \ScZ(s,\varphi_\pi^\circ,\BL_\rho)=L(s,\pi,\rho).
    \]
    \item {\bf Fourier Transform:} There exists an invariant distribution $\Phi_{\rho,\psi}$ on $G(F)$ with a non-trivial additive character $\psi$ of $F$ such that the convolution operator 
    \[
    \ScF_{\rho,\psi}(\phi)(g)=|\Fd(g)|_F^{-n_\rho}(\Phi_{\rho,\psi}*\phi^\vee)(g),\quad {\rm with}\ \phi\in\CS_\rho(G(F))\ {\rm and}\ g\in G(F)
    \]
    defines a Fourier transform on $\CS_\rho(G(F))$, which enjoys the following properties: 
    $\ScF_{\rho,\phi}(\BL_\rho)=\BL_\rho$, $\ScF_{\rho,\psi}\circ\ScF_{\rho,\psi^{-1}}=\RI$, the identity operator, and 
    $\ScF_{\rho,\psi}$ extends to a unitary operator on $L^2(G(F),\ud_\rho g)$, 
    where $\ud_\rho(g)=|\Fd(g)|_F^{n_\rho}\ud g$ and $\phi^\vee(g)=\phi(g^{-1})$.
    \item {\bf Local $\gam$-Factor:} The local functional equation 
    \[
    \ScZ(1-s,\varphi_\pi^\vee,\ScF_{\rho,\psi}(\phi))
    =\gam(s,\pi,\rho,\psi)\ScZ(s,\varphi_\pi,\phi) ,\quad 
    {\rm for}\ \varphi_\pi\in\CC(\pi)\ {\rm and}\ \phi\in\CS_\rho(G(F)),
    \]
    holds with the Langlands local $\gam$-factor $\gam(s,\pi,\rho,\psi)$, as meromorphic functions in $s\in\BC$.
\end{enumerate}
\end{cnj}

It is clear that the $\bkn$-triples are the right set-up to extend the Godement-Jacquet zeta theory to such a great generality.  On the other hand, it is also clear that not every triple in the definition of the Langlands automorphic $L$-functions satisfies the conditions in Definition \ref{dfn:BKN}. Without loss of generality, 
in terms of automorphic $L$-functions, we may introduce the notion of $L$-triples. 

\begin{dfn}[$L$-Triple]\label{dfn:LT}
Let $k$ be a field of characteristic zero. Let $G$ be a reductive $k$-split algebraic group, $G^\vee$ be the Langlands dual group, which is also a reductive $k$-split algebraic group, and $\rho$ be an irreducible finite-dimensional $k$-rational representation of $G^\vee$. The triple $(G,G^\vee,\rho)$ is called an {\bf $L$-triple} if the restriction of 
$\rho$ to any essentially simple component of the derived group of $G^\vee$ is non-trivial. 
\end{dfn}

In the notes, we are going to show that in terms of the Langlands local $L$-factors, the BKN-triples catch the great generality. 
More precisely, for any given $L$-triple $(G,G^\vee,\rho)$ as in Definition \ref{dfn:LT}, one may construct a canonical triple $\Delta_\fp=(H_\rho,H_\rho^\vee,\rho_\fp)_{\Fd_\fp}$ via a canonical fiber product construction associated to $(G,G^\vee,\rho)$, such that the following holds. 

\begin{thm}\label{thm:localm}
Let $F$ be a local field of characteristic zero. 
For any given $L$-triple $(G,G^\vee,\rho)$ with $G$ being $F$-split, there is a {\rm BKN}-triple  
$\Delta_{\bkn}(\Fd)=(G_\Fd,G_\Fd^\vee,\rho^\Fd)$ such that 
for any given irreducible admissible representation $\pi$ of 
    $G(F)$, there exists an irreducible admissible representation $\sig$ of $G_\Fd(F)$ such that 
    the identification of local $L$-factors:
    \[
    L(s,\pi,\rho)=L(s,\sig,\rho^\Fd)
    \]
    holds as meromorphic functions in $s\in\BC$. Note that when $F$ is archimedean, the representations 
    $\pi$ and $\sig$ are of the Casselman-Wallach type. The local $L$-factors are defined by assuming that the local 
    Langlands conjecture for both triples over $F$.
\end{thm}

Theorem \ref{thm:localm} is one of the main results of this paper, which will be proved in Section \ref{sec-BKNT}, 
with the BKN-triple assertion proved in Section \ref{ssec-BKNT-FP} (Theorem \ref{thm:BKNT}) and the local $L$-factor assertion proved in Section \ref{ssec-LLF}. With the help of the Borel conjecture (\cite{Bor79,AP19, Sol20}), 
one can make the representation $\sig$ more explicitly in terms of the given representation $\pi$ in Theorem \ref{thm:localm}, which will be discussion in Section \ref{sec-BC}. Theorem \ref{thm:localm} suggests 
that in order to study the Langlands conjecture on general automorphic $L$-functions via global zeta integral methods, 
which include the Rankin-Selberg method, the Langlands-Shahidi method, the Miller-Schmid method, and the Bravermen-Kazhdan-Ng\^o method and so on, the harmonic analysis can be taken over the BKN-triples that have the extra nice structures as in Definition \ref{dfn:BKN}. This will be further explained by a series of examples in Section \ref{sec-Ex}, and in our forthcoming work. 

As explained in \cite{BK00, Ngo20, BNS16}, the Braverman-Kazhdan-Ng\^o approach to the Langlands conjecture for automorphic $L$-functions and other related geometric problems needs a good understanding of the $L$-monoids associated with the group $G$. For a given BKN-triple $(G,G^\vee,\rho)_\Fd$, one has two different constructions 
of $L$-monoids: one is from the Vinberg method (\cite{Vin95, Ngo, Ngo14}), which is denoted by $\CM_\rho$, and the other is from the general theory of M. Putcha and L. Renner (\cite{Ren05}), which is denoted by $\CM^\rho$. Note that 
the two methods construct the monoids with $G$ as the unit groups under the different assumptions on the structure of the algebraic groups $G$. When $G$ is from a BKN-triple, the two methods work and hence we have both $\CM_\rho$ and 
$\CM^\rho$. It is important to know whether $\CM_\rho\cong\CM^\rho$ since both have different geometric features. 

\begin{thm}\label{thm:monoids}
    For any given $\bkn$-triple $\Delta_{\bkn}(\Fd):=(G,G^\vee,\rho)_\Fd$, let $\CM_\rho$ be the $L$-monoid as constructed from the Vinberg method and let $\CM^\rho$ be the $L$-monoid as constructed from the Putcha-Renner method. Then the two reductive monoids are isomorphic, i.e., $\CM_\rho\cong\CM^\rho$. 
\end{thm}

The proof of Theorem \ref{thm:monoids} will be given in Section \ref{ssec-PThmM}. It is clear that 
Theorem \ref{thm:monoids} provides a crucial geometric base for further harmonic analysis over the $L$-monoids for automorphic $L$-functions and other related problems (\cite{BK00, Ngo20, BNS16}). 

\quad

This paper is organized as follows. Section \ref{sec-BKNT} devotes to the construction of the BKN-triple associated with any given $L$-triple and proves Theorem \ref{thm:localm}. Section \ref{sec-LMBKN} is to discuss the relation between the monoid constructed via the Putcha and Renner theory and the monoid constructed through the Vinberg universal monoid for any given BKN-triple. In Section \ref{sec-Ex}, we present a series of examples of the BKN-triples. Finally, in Section \ref{sec-BC}, we discuss the Borel conjecture on the local $L$-packets, which are important for us to understand the known construction of global zeta integrals for families of $L$-functions.

%%%%%%%%%%%%%%%%%%%%
\section{Braverman-Kazhdan-Ng\^o Triple}\label{sec-BKNT}
%%%%%%%%%%%%%%%%%%%

Let $k$ be a field of characteristic zero, which can be a number field or a local field of characteristic zero as 
considered in this paper. Let $G$ be a reductive algebraic group that is $k$-split. Take $G^\vee$ to be the Langlands dual group that is considered also as a reductive $k$-split algebraic group. Let $\rho$ be a finite-dimensional 
$k$-rational representation of $G^\vee$. For a given $L$-triple $(G, G^\vee,\rho)$, Section \ref{ssec-BKNT-FP} shows that 
a triple defined via the fiber product associated with $(G, G^\vee,\rho)$ is a BKN-triple (Theorem \ref{thm:BKNT}). 
Note that for the consideration of the Langlands $L$-functions, we usually take $\rho$ to 
be a representation of $G^\vee(\BC)$. However, for the construction of the BKN-triple from the given $L$-triple 
it may be more convenient to take $\rho$ to be a rational representation of $G^\vee$, especially, when we assume that 
$G$ and $G^\vee$ are split reductive groups over $k$. Following the spirit of the relative Langlands duality program of D. Ben-Zvi, Y. Sakellaridis, and A. Venkatesh (\cite{BZSV}), such a set-up has its own advantage. 

For the given triple over $k$, let $T$ be a maximal $k$-split torus of $G$ and $T^\vee$ be its dual, which is a maximal $k$-split torus of $G^\vee$. Let $X^*(T)$ be the $F$-rational characters of $T$ and $X_*(T)$ be the $F$-rational co-characters of $T$. Then we have that $X^*(T)=X_*(T^\vee)$ and $X^*(T^\vee)=X_*(T)$. Let $\Phi=\Phi(G,T)$ 
be the set of roots of $G$ with respect to $T$ and $\Phi^\vee=\Phi(G^\vee,T^\vee)$ be the set of roots of $G^\vee$ with respect to $T^\vee$, which is dual to $\Phi$. Let $(X,\Phi;X^\vee,\Phi^\vee)$ be the root data as in \cite{Spr98}.
Let $(B,T)$ be a fixed Borel pair of $G$ that determines the subset $\Phi^+$ of positive roots in $\Phi$. 

\subsection{BKN-triple via fiber product}\label{ssec-BKNT-FP}
For a given $L$-triple $(G,G^\vee,\rho)$ as in Definition \ref{dfn:LT}, in order to prove Theorem \ref{thm:localm}, we are going to construct a BKN-triple 
\[
\Delta_{\bkn}(\Fd)=(G_\Fd,G_\Fd^\vee,\rho^\Fd)
\]
via the fiber product associated to $(G,G^\vee,\rho)$. 
We start with a construction of $G_\Fd^\vee$ and consider the following fiber product diagram:
\begin{align}\label{diag-fp1}
    \begin{matrix}
        G^\vee&\stackrel{\rho}\longrightarrow &\GL_{d_\rho}\\
        \downarrow\scriptstyle{\pr}&& \downarrow\scriptstyle{\pr}\\
        (G^\vee)_\ad&\stackrel{\ol{\rho}}\longrightarrow &\PGL_{d_\rho}
    \end{matrix},
\end{align}
where $d_\rho$ is the dimension of the representation $\rho$ and $\ol\rho$ is the canonical morphism induced 
from $\rho$ via the canonical projection. 
Let $H_\rho^\vee$ be the fiber product of $\GL_{d_{\rho}}$ and $(G^{\vee})_{\ad}$ over $\PGL_{d_{\rho}}$ defined by the diagram in \eqref{diag-fp1}, which is the universal object fits into the following commutative diagram:
\begin{align}\label{diag-fp2}
    \begin{matrix}
        & &\GL_{d_\rho}& &\\
        &&&&\\
        &\nearrow\scriptstyle{\rho}&\uparrow\scriptstyle{\rho_\fp}&\scriptstyle{\pr}\searrow&\\
        &&&&\\
        G^\vee&\stackrel{\eta_\fp^\vee}\longrightarrow&H_\rho^\vee&&\PGL_{d_\rho}\\
        &&&&\\
        &\scriptstyle{\pr}\searrow&\downarrow\scriptstyle{\pr}&\nearrow\scriptstyle{\ol{\rho}}&\\
        &&&&\\
        & &(G^\vee)_\ad&&
    \end{matrix}
\end{align}
The morphisms 
\[
\eta_\fp^\vee\colon G^\vee\longrightarrow H_\rho^\vee,\quad\rho_\fp\colon H_\rho^\vee\longrightarrow\GL_{d_\rho}, \quad {\rm and}\quad \pr\colon H^\vee_\rho\longrightarrow (G^\vee)_\ad
\]
are given by the universality of the fiber product, so that \eqref{diag-fp2} is a commutative diagram. 
Note that Diagram \eqref{diag-fp2} also implies the following diagram:
\begin{align}\label{diag-fp3}
    \begin{matrix}
    1&\longrightarrow&\BG_m&\stackrel{\Fd_\fp^\vee}\longrightarrow&H_\rho^\vee&\stackrel{\pr}\longrightarrow&(G^\vee)_\ad&\longrightarrow&1\\
    &&&&&&&&\\
    &&\downarrow\scriptstyle{=}&&\downarrow\scriptstyle{\rho_\fp}&&\downarrow\scriptstyle{\ol{\rho}}&\\
    &&&&&&&&\\
    1&\longrightarrow&\BG_m&\longrightarrow&\GL_{d_\rho}&\stackrel{\pr}\longrightarrow&
    \PGL_{d_\rho}&\longrightarrow&1\\
        \end{matrix}
\end{align}
which implies that the composition $\rho_\fp\circ\Fd_\fp^\vee(z)=z\RI_{d_\rho}$ for any $z\in\BG_m$, which is 
Condition (4) of Definition \ref{dfn:BKN}.

Now we take $H_\rho$ to be the Langlands dual of $H_\rho^\vee$ and it is equipped with a morphism 
\begin{align}\label{Fdfp}
    \Fd_\fp\colon H_\rho\longrightarrow\BG_m,
\end{align}
which is dual to $\Fd_\fp^\vee\colon \BG_m\longrightarrow H_\rho^\vee$ as in \eqref{diag-fp3}. In this way, we construct a triple 
which is denoted by 
\begin{align}\label{Dfp}
    \Delta_\fp=(H_\rho,H_\rho^\vee,\rho_\fp)_{\Fd_\fp}. 
\end{align}
In order to show that the constructed triple $\Delta_\fp$ is a BKN-triple, i.e. $\Delta_\fp=\Delta_\bkn(\Fd_\fp)$, it remains to show that the following 
exact sequence 
\begin{align}\label{es-1}
    1\longrightarrow H_\rho^\circ\longrightarrow H_\rho\stackrel{\Fd_\fp}\longrightarrow\BG_m\longrightarrow1
\end{align}
holds with the kernel $H_\rho^\circ$ of the morphism $\Fd_\fp$ equals the derived group of $H_\rho$, i.e. $H_\rho^\circ=[H_\rho,H_\rho]$, 
the kernel $\ker(\rho_\fp)$ is trivial, and $H_\rho^\circ$ is simply connected according to Definition \ref{dfn:BKN}. 

\begin{prp}\label{prp:kernelfp}
    With notations as given by the fiber product diagram in \eqref{diag-fp2}, the kernel of the morphism $\rho_\fp\colon H_\rho^\vee\longrightarrow\GL_{d_\rho}$ is trivial. 
\end{prp}

\begin{proof}
Assume that the irreducible representation $\rho=\rho_\lam$ of $G^\vee$ has its dominant weight $\lam\in X^*(T^\vee)$. 
Let $Z^\vee$ be the center of $G^\vee$. Then $\lam\colon Z^\vee\longrightarrow\BG_m$ is a character of $Z^\vee$. 
       We claim that
    \begin{align}\label{claim1}
        H_\rho^\vee\cong(\BG_m \times G^\vee) / \{ (\lambda(z)^{-1},z) \mid z \in Z^\vee \}=:
        \BG_m\times_{(\lam^{-1},\Id)(Z^\vee)}G^\vee.
    \end{align}

In order to prove the isomorphism in \eqref{claim1}, it suffices to show that the group on the right-hand side 
    $\BG_m\times_{(\lam^{-1},\Id)(Z^\vee)}G^\vee$ is also 
    a universal object for the fiber product diagram in \eqref{diag-fp1}. 
    First of all, we extend the representation $\rho$ of $G^\vee$ to that of the group $\BG_m \times G^\vee$ by 
    defining the scalar multiplication of $\BG_m$:
    \begin{align}\label{rhoe1}
        \rho^\Fe(a,g)=a\cdot\rho(g),\quad {\rm for}\ a\in\BG_m,\ {\rm and}\ g\in G^\vee.
    \end{align}
    When $(a,g)=(\lam^{-1}(z),z)\in\BG_m\times G^\vee$ for $z\in Z^\vee$, we have that 
    \[
    \rho^\Fe((\lam^{-1}(z),z))=\lam^{-1}(z)\cdot\rho(z)=1.
    \]
Then the extended representation $\rho^\Fe$ factors through the quotient group $\BG_m\times_{(\lam^{-1},\Id)(Z^\vee)}G^\vee$:
\begin{align}\label{rhoFe}
   \BG_m\times_{(\lam^{-1},\Id)(Z^\vee)}G^\vee\stackrel{\rho^\Fe}\longrightarrow\GL_{d_{\rho}}
 \colon [(a,g)]\mapsto a\cdot \rho(g). 
\end{align}
On the other hand, we have the natural project:
    \begin{align}\label{pr2}
      \BG_m\times_{(\lam^{-1},\Id)(Z^\vee)}G^\vee
    \stackrel{\pr}\longrightarrow(G^{\vee})_{\ad}\colon [(a,g)]\mapsto [g].  
    \end{align}
It is clear that the above construction fits into the fiber product diagram as in \eqref{diag-fp1}:
\begin{align}\label{diag-fp4}
\begin{matrix}
        \BG_m\times_{(\lam^{-1},\Id)(Z^\vee)}G^\vee&\stackrel{\rho^\Fe}\longrightarrow &\GL_{d_\rho}\\
        \downarrow\scriptstyle{\pr}&& \downarrow\scriptstyle{\pr}\\
        (G^\vee)_\ad&\stackrel{\ol{\rho}}\longrightarrow &\PGL_{d_\rho}
    \end{matrix}.
\end{align}
It remains to show the universality of the group $\BG_m\times_{(\lam^{-1},\Id)(Z^\vee)}G^\vee$ with respect to the fiber product diagram in \eqref{diag-fp4}.

Let $X$ be any variety fitting into the commutative diagram
\begin{equation}\label{diag-fp5}
\begin{array}[c]{ccc}
X&\stackrel{f_1}\longrightarrow&\GL_{d_{\rho}}\\
\downarrow\scriptstyle{f_2}&&\downarrow\scriptstyle{\pr}\\
(G^{\vee})_{\ad}&\stackrel{\ol{\rho}}{\longrightarrow}&\PGL_{d_{\rho}}.
\end{array}   
\end{equation}
We have to show that there exists a unique morphism $f\colon X\longrightarrow\BG_m\times_{(\lam^{-1},\Id)(Z^\vee)}G^\vee$
such that the following commutative diagram 
\begin{align}\label{diag-fp6}
    \begin{matrix}
        & &\GL_{d_\rho}& &\\
        &&&&\\
        &\nearrow\scriptstyle{\rho^\Fe}&\uparrow\scriptstyle{f_1}&\scriptstyle{\pr}\searrow&\\
        &&&&\\
\BG_m\times_{(\lam^{-1},\Id)(Z^\vee)}G^\vee&\stackrel{f}\longleftarrow&X&&\PGL_{d_\rho}\\
        &&&&\\
        &\scriptstyle{\pr}\searrow&\downarrow\scriptstyle{f_2}&\nearrow\scriptstyle{\ol{\rho}}&\\
        &&&&\\
        & &(G^\vee)_\ad&&
    \end{matrix}
\end{align}
We first construct the map $f$ and then prove that it is an algebraic morphism. 

We first consider the morphism $f_2\colon X\longrightarrow(G^\vee)_\ad$ and the short exact sequence 
\[
1\rightarrow Z^\vee\rightarrow G^\vee\stackrel{\pr}\rightarrow(G^\vee)_\ad\rightarrow1.
\]
For any $x\in X$, we take a local section $\Fs_x$ of the projection $\pr$ at $f_2(x)$, namely, take 
an open affine neighborhood $U_x$ of $f_2(x)$ together with a morphism 
$\Fs_x\colon U_x\longrightarrow V_x\subset G^{\vee}$, where $V_x$ is some open affine subvariety of $G^{\vee}$, such that $(\pr\circ\Fs_x)(y)=y$ for any $y\in U_x$. Let $W_x$ be an open affine neighborhood of $x$ in $X$ such that $f_2(W_x)\subset U_x$ since $f_2$ is a morphism. 

Write $g_x=\Fs_x(f_2(x))\in G^\vee$.
According to the commutative diagram in \eqref{diag-fp6}, there is some $a_x\in\BG_m$ such that 
\begin{align}\label{ax}
    f_1(x)=a_x\rho(g_x)=\rho^\Fe(a_x,g_x),\quad {\rm with}\ (a_x,g_x)\in\BG_m\times G^\vee. 
\end{align}
Since the local section is not unique, if we take another local section $\Fs_x'$ at $f_2(x)$ with similar open affine 
neighborhoods, then we have that $g_x'=\Fs_x'(x)$ and 
\[
f_1(x)=a'_x\rho(g'_x)=\rho^\Fe(a'_x,g'_x),\quad {\rm with}\ (a'_x,g'_x)\in\BG_m\times G^\vee. 
\]
Hence we obtain that 
\[
f_1(x)=a_x\rho(g_x)=a_x'\rho(g_x'),
\]
which can be rewritten as 
\[
\rho(g_x(g'_x)^{-1})=a_x^{-1}a'_x.
\]
Since $(G,G^\vee,\rho)$ is an $L$-triple as in Definition \ref{dfn:LT}, and $\pr\circ\rho=\ol{\rho}\circ\pr$ by \eqref{diag-fp1}, there is some $z\in Z^{\vee}$ such that 
$g_x(g'_x)^{-1}=z$, which implies that 
\[
(a_x,g_x)=(a'_x,g'_x)(\lambda(z)^{-1},z).
\]
Hence the image of $(a_x,g_x)\in\BG_m\times G^{\vee}$ in the quotient group 
$\BG_m\times_{(\lam^{-1},\Id)(Z^\vee)}G^\vee$, which is denoted by $[a_x,g_x]$, 
does not depend on the choice of the local sections and is uniquely determined by the given $x\in X$. Therefore, 
we construct a map 
\begin{align}\label{mapf}
    x\mapsto f(x)=[a_x,g_x]
\end{align}
from $X$ to the quotient group $\BG_m\times_{(\lam^{-1},\Id)(Z^\vee)}G^\vee$. It is clear from the construction that 
the map $f$ fits into the commutative diagram in \eqref{diag-fp6}. Moreover, from the construction of the map $f$, we have 
\[
g_x=\Fs_x(f_2(x))\quad {\rm and}\quad 
a_x=f_1(x)\cdot \rho(g_x)^{-1}=f_1(x)\cdot \rho(\Fs_x(x))^{-1}
\]
according to \eqref{ax}. Hence we obtain that the map $f$ is locally explicitly given by 
\[
x\mapsto f(x)=[a_x,g_x]=[f_1(x)\cdot \rho(\Fs_x(x))^{-1},\Fs_x(f_2(x))]
\]
which is clearly algebraic and thus, a morphism. 

It remains to prove the uniqueness of the morphism $f\colon X\longrightarrow\BG_m\times_{(\lam^{-1},\Id)(Z^\vee)}G^\vee$ 
that fits into the commutative diagram in \eqref{diag-fp6}. 
Suppose that 
\[
X\stackrel{f^{\prime}}\longrightarrow\BG_m\times_{(\lam^{-1},\Id)(Z^\vee)}G^\vee\colon x\mapsto[b_x,h_x]
\]
is another morphism fitting into the commutative diagram in \eqref{diag-fp6}, where $b_x\in\BG_m$ and $h_x\in G^{\vee}$. For any $x\in X$, since $\pr(h_x)=[h_x]=f_2(x)\in (G^{\vee})_{\ad}$, $h_x\in G^{\vee}$ is a pre-image of $f_2(x)$ (with a choice of a section). With $h_x$ chosen, we must have that $b_x$ satisfies the equlation:
\[
b_x\rho(h_x)=\rho^\Fe(b_x,h_x)=f_1(x)
\]
according to the commutative diagram in \eqref{diag-fp6}. From the discussion above, the morphism $f$ is independent 
of the choice of the local sections. Hence we must have that $f=f'$, which proves the uniqueness of the morphism $f$.

Finally, it is easy to verify that the kernel of the morphism $\rho_\fp\colon H_\rho^\vee\longrightarrow\GL_{d_\rho}$ is trivial from the following commutative diagram:
\begin{align}\label{diag-fp7}
    \begin{matrix}
        & &\GL_{d_\rho}& &\\
        &&&&\\
        &\nearrow\scriptstyle{\rho^\Fe}&\uparrow\scriptstyle{\rho_\fp}&\scriptstyle{\pr}\searrow&\\
        &&&&\\
\BG_m\times_{(\lam^{-1},\Id)(Z^\vee)}G^\vee&\stackrel{\vsig}\cong&H_\rho^\vee&&\PGL_{d_\rho}\\
        &&&&\\
        &\scriptstyle{\pr}\searrow&\downarrow\scriptstyle{\pr}&\nearrow\scriptstyle{\ol{\rho}}&\\
        &&&&\\
        & &(G^\vee)_\ad&&
    \end{matrix}
\end{align}
In fact, if $h\in H_\rho^\vee$ belongs to the kernel $\ker(\rho_\fp)$, then $f(h)=[a,g]$ belongs to the $\ker(\rho^\Fe)$. This means that $1=\rho^\Fe(a,g)=a\rho(g)$ and $\rho(g)=a^{-1}\in\BG_m$. Since $(G,G^\vee,\rho)$ 
is an $L$-triple, we must obtain that $g=z\in Z^\vee$, which implies that $a=\rho(z)^{-1}=\lam(z)^{-1}$ and 
$[a,g]=[\lam(z)^{-1},z]$ must be the identity element of $\BG_m\times_{(\lam^{-1},\Id)(Z^\vee)}G^\vee$. Therefore, 
the given element $h$ must be the identity element in $H_\rho^\vee$. We are done. 
\end{proof}

From the proof of Proposition \ref{prp:kernelfp}, we obtain the following structure of the group $H_\rho^\vee$.

\begin{cor}\label{cor:Hrhod}
    The isomorphism $H_\rho^\vee\stackrel{\vsig}\cong\BG_m\times_{(\lam^{-1},\Id)(Z^\vee)}G^\vee$ as in the commutative diagram in \eqref{diag-fp7} is explicitly given as follows:
    \[
    h\in H_\rho^\vee\mapsto\vsig(h)=[\rho_\fp(h)\cdot\rho(\Fs_h(h))^{-1},\Fs_h(\pr(h))]\in 
    \BG_m\times_{(\lam^{-1},\Id)(Z^\vee)}G^\vee,
    \]
    which is independent of the choice of the local section $\Fs_h$ near $\pr(h)$ from $(G^\vee)_\ad$ to $G^\vee$.
\end{cor}

\begin{prp}\label{prp:es}
    The group $H_\rho^\circ$ defined via \eqref{es-1} is the derived group of $H_\rho$, i.e. $H_\rho^\circ=[H_\rho,H_\rho]$.
\end{prp}

\begin{proof}
We first consider the structure of the dual group $H_\rho^\vee$ by the isomorphism:
\[
H_\rho^\vee\stackrel{\vsig}\cong\BG_m\times_{(\lam^{-1},\Id)(Z^\vee)}G^\vee,
\]
which is given in Corollary \ref{cor:Hrhod} and the proof of Proposition \ref{prp:kernelfp}. The given maximal torus of $\BG_m\times_{(\lam^{-1},\Id)(Z^\vee)}G^\vee$ is $\BG_m\times_{(\lam^{-1},\Id)(Z^\vee)}T^\vee$, which has the following isomorphism:
\[
\BG_m\times_{(\lam^{-1},\Id)(Z^\vee)}T^\vee\cong \BG_m \times (T^\vee)_\ad \colon [a,t] \mapsto (\lambda(t)a,[t]).
    \]
   Let $(X,\Phi, X^{\vee},\Phi^{\vee})$ be the root datum of $G$ and denote the root lattice, the coroot lattice, and the weight lattice by $Q$, $Q^\vee$ and $P$ respectively. We may denote by $\alpha \mapsto \bar{\alpha}$ 
   the map $X \to P$, which is dual to $Q^\vee \hookrightarrow X^\vee$, and identify the maximal split torus of $H_\rho^\vee$ with $\BG_m \times (T^\vee)_\ad$. Then the root datum of $H_\rho^\vee$ is given by
    \[
    (\mathbb{Z}\oplus Q^\vee, \{0\}\oplus\Phi^\vee, \mathbb{Z}\oplus P, \{\langle \alpha,\lambda\rangle ,\alpha)\mid \alpha\in\Phi \} ). 
    \]
    Associated with the  natural morphism $G^\vee \to H_\rho^\vee$, one has the map between the root data given by
    \[
         \mathbb{Z}\oplus Q^\vee \to X^\vee : (n, \alpha^\vee) \mapsto n\lambda+\alpha^\vee,
    \]
    with its dual
    \[
        X \to \mathbb{Z}\oplus P : \alpha \mapsto (\langle\alpha,\lambda\rangle, \bar{\alpha}).
    \]

    On the other hand, let $G_0^\rsc$ be the simply-connected cover of $G_0$. We consider the semi-direct product $\BG_m\ltimes_{\lambda}G_0^\rsc$ with the multiplication given by
    \begin{align}\label{mf0}
     (a_1,g_1)(a_2,g_2)=(a_1a_2,\lambda(a_2)^{-1}g_1\lam(a_2)g_2),\quad {\rm for}\ a_1,a_2\in\BG_m, g_1,g_2\in G_0^\rsc.   
    \end{align}
    Here the highest weight $\lam\in X^*(T^\vee)$ is used as a co-character $\lambda\in X_*(T)$ of $T$. We may use the same $\lam$ to denote its image in $X_*(T_\ad)$ via the projection from $T$ to $T_\ad$. Since $G$ and $G_0^\rsc$ share the same adjoint group, we have that $T_0^\rsc / Z_0^\rsc = T_\ad$, where $T_0^{\rsc}$ is the maximal torus of $G_0^{\rsc}$ and $Z_0^{\rsc}$ is the center of $G_0^{\rsc}$. Hence the formula for the multiplication 
    in \eqref{mf0} is well-defined.
Take the maximal split torus $\BG_m \ltimes_\lam T_0^\rsc$ of $\BG_m\ltimes_{\lambda}G_0^{\rsc}$. Since $\lam(\BG_m)$ 
commutes with $T_0^\rsc$, we have that $\BG_m \ltimes_\lam T_0^\rsc=\BG_m \times T_0^\rsc$. 
Thus, the root datum of $\BG_m\ltimes_{\lambda}G_0^\rsc$ with respect to $\BG_m \times T_0^\rsc$ is given by
    \[
    (\mathbb{Z}\oplus P, \{\langle \alpha,\lambda\rangle ,\alpha)\mid \alpha\in\Phi \} , \mathbb{Z}\oplus Q^\vee, \{0\}\oplus\Phi^\vee).
    \]

    Denote the covering $G_0^\rsc \to G_0$ by $g \mapsto\bar g$, and consider the homomorphism
    \begin{align}\label{map0}
        \BG_m\ltimes_{\lambda}G_0^\rsc \longrightarrow G : (a,g) \mapsto \lambda(a)\bar g.
    \end{align}  
    Then the associated map on the root data is given by
    \[
        X \to \mathbb{Z}\oplus P : \alpha \mapsto (\langle\alpha,\lambda\rangle, \bar{\alpha}),
    \]
    with its dual
    \[
        \mathbb{Z}\oplus Q^\vee \to X^\vee : (n, \alpha^\vee) \mapsto n\lambda+\alpha^\vee.
    \]
    The above proves that $H_\rho\cong\BG_m\ltimes_{\lambda}G_0^\rsc$. Hence the derived subgroup $[H_\rho,H_\rho]$ is 
    isomorphic to the derived subgroup of $\BG_m\ltimes_{\lambda}G_0^\rsc$, which is clearly $\{ 1 \} \times G_0^\rsc$.

    Finally, from the above discussion, the projection 
    \[
        \BG_m\times_{(\lam^{-1},\Id)(Z^\vee)}G^\vee=\BG_m \times G^\vee / \{ (\lambda(z)^{-1},z) \mid z \in Z^\vee \} \longrightarrow (G^\vee)_\ad : [a,g] \mapsto [g]
    \]
    is the dual of the embedding
    \[
        G_0^\rsc \longrightarrow \BG_m\ltimes_{\lambda}G_0^\rsc : g \mapsto (1,g),
    \]
    which implies the kernel $H_\rho^\circ$ is also the subgroup $\{ 1 \} \times G_0^\rsc$. Therefore we prove that   $H_\rho^\circ\cong[H_\rho,H_\rho]$ is simply connected.
\end{proof}

From the proof of Proposition \ref{prp:es}, we obtain the following structure of the group $H_\rho$.

\begin{cor}\label{cor:Hrhosc}
    The group $H_\rho$, which is defined as the Langlands dual of the group $H_\rho^\vee$ as in \eqref{diag-fp2} 
    has the following property.
    \begin{enumerate}
        \item $H_\rho\cong\BG_m\ltimes_{\lambda}G_0^\rsc$, whose multiplication is given as in \eqref{mf0}.
        \item Its derived group $H_\rho^\circ=[H_\rho,H_\rho]\cong G_0^\rsc$ is simply connected.
    \end{enumerate} 
\end{cor}

By the above discussion with Propositions \ref{prp:kernelfp} and \ref{prp:es} and Corollary \ref{cor:Hrhosc}, we obtain the following result, which implies the first part of Theorem \ref{thm:localm}.

\begin{thm}\label{thm:BKNT}
Let $k$ be any field of characteristic zero. For any given $L$-triple $(G,G^\vee,\rho)$ with the reductive algebraic group $G$ being $F$-split, the constructed tripe $(H_\rho,H_\rho^\vee,\rho_\fp)$ is a 
{\rm BKN}-triple, i.e., $\Delta_\fp=(H_\rho,H_\rho^\vee,\rho_\fp)_{\Fd_\fp}=\Delta_\bkn(\Fd_\fp)$.
\end{thm}

\subsection{Local $L$-factor}\label{ssec-LLF}
To complete the proof of Theorem \ref{thm:localm}, we consider local $L$-factors in the sense of Langlands. 

Over a local field $F$ of characteristic zero we prove a relation of local $L$-factors, whose definitions are 
given under the assumption that the local Langlands conjecture holds for the relevant group $G(F)$. 
More precisely, they are defined by 
\begin{align}\label{LLF}
L(s,\pi,\rho_\lambda):=L(s,\rho_\lambda\circ\phi_\pi)
\end{align}
if $\phi_\pi$ is the local $L$-parameter of $\pi$. Note that this definition of local $L$-factors is of assumption-free when either $\pi$ is unramified or $F$ is archimedean. 
When $F$ is non-archimedean and $\pi$ is ramified, the local Langlands conjecture for $(G,F)$ remains to be proved in general, although many cases have been proved.

To complete the proof of Theorem \ref{thm:localm}, we take $F$ to be any local field of characteristic zero. From the fiber product diagram in \eqref{diag-fp2}, there is a morphism 
\begin{align}\label{etafpvee}
    \eta_\fp^\vee\colon G^\vee\longrightarrow H_\rho^\vee.
\end{align}
Let $\ScL_F$ be the local Langlands group of $F$. Then the local $L$-parameters of $G(F)$ are the admissible 
homomorphisms 
\[
\phi\colon \ScL_F\longrightarrow G^\vee(\BC)
\]
up to conjugation by $G^\vee(\BC)$. Let $\Pi_F(G)$ be the set of equivalence classes of irreducible admissible representations of $G(F)$, which are assumed to be of the Casselman-Wallach type if $F$ is archimedean. 

By the assumption in Theorem \ref{thm:localm}, the local Langlands conjecture holds for both triples $(G,G^\vee,\rho)$ and $\Delta_\fp=(H_\rho,H_\rho^\vee,\rho_\fp)_{\Fd_\fp}$ over $F$. For any $\pi\in\Pi_F(G)$, 
there is a local Langlands parameter $\phi_\pi$ associated with $\pi$. It is clear that the composition 
$\eta_\fp^\vee\circ\phi_\pi$ is a local Langlands parameter for the group $H_\rho(F)$. Hence by the local Langlands 
conjecture for $H_\rho(F)$, we obtain that $\eta_\fp^\vee\circ\phi_\pi$ is the local Langlands parameter for some 
$\sig\in\Pi_F(H_\rho)$. In other words, we have the following diagram:
\begin{align}\label{diag-L1}
    \begin{matrix}
         \ScL_F &&\stackrel{\phi_\pi}\longrightarrow&&G^\vee(\BC)&&\stackrel{\rho}\longrightarrow&&\GL_{d_\rho}(\BC)\\ 
              &&&&&&&&\\ 
        &&\scriptstyle{\phi_\sig}\searrow&&\downarrow\scriptstyle{\eta_\fp^\vee}&&\nearrow\scriptstyle{\rho_\fp}&&\\
        &&&&&&&&\\
        &&&&H_\rho^\vee(\BC)&&&&
    \end{matrix}
\end{align}
with $\phi_\sig=\eta_\fp^\vee\circ\phi_\pi$. 
According to the definition of local $L$-factors in \eqref{LLF}, we obtain that 
\[
L(s,\pi,\rho)=L(s,\rho\circ\phi_\pi)=L(s,\rho_\fp\circ\eta^\vee_\fp\circ\phi_\pi)=L(s,\rho_\fp\circ\phi_\sig)=L(s,\sig,\rho_\fp). 
\]
Combining with Theorem \ref{thm:BKNT}, this completes the proof of Theorem \ref{thm:localm}.

%%%%%%%%%%%%%%%%
\section{$L$-Monoid and BKN-Triple}\label{sec-LMBKN}
%%%%%%%%%%%%%%%%%%

By Theorem \ref{thm:localm}, for any given $L$-triple $(G,G^\vee,\rho)$, in order to understand the Langlands $L$-function associated with $(G,G^\vee,\rho)$, locally and globally, it is enough to work with the associated BKN-triple
$\Delta_\fp=(H_\rho,H_\rho^\vee,\rho_\fp)_{\Fd_\fp}$, which is a triple given by Definition \ref{dfn:BKN}. In this section, we may take any BKN-triple $\Delta_{\bkn}(\Fd):=(G,G^\vee,\rho)_\Fd$ from Definition \ref{dfn:BKN}.

As explained in \cite{Ngo20}, the Schwartz space in the Braverman-Kazhdan proposal may depend on the geometry of the 
$L$-monoid $\CM$ that contains $G$ as the unit group, which is the open subset of all invertible elements. 
From \cite[Proposition 5.1]{Ngo20}, one can deduce the existence of such $L$-monoid $\CM^\rho$, following the general theory of M. Putcha and L. Renner (\cite{Ren05}). For the 
BKN-triple $\Delta_{\bkn}(\Fd):=(G,G^\vee,\rho)_\Fd$, since the derived group $G_0=[G,G]$ is simply connected, one may also construct a monoid $\CM_\rho$ from the construction of the Vinberg universal monoid (\cite{Vin95, Ngo, Ngo14}) that also contains $G$ as the unit group. The goal of this section is to prove Theorem \ref{thm:monoids}, which will be done in Section \ref{ssec-PThmM}

\subsection{Construction via the Vinberg method}\label{ssec-CVM}

For a $\bkn$-triple $\Delta_{\bkn}(\Fd):=(G,G^\vee,\rho)_\Fd$, we construct the $L$-monoid 
$\CM_\rho$ via the Vinberg method. 

Take $k$ to be a field of characteristic zero and fix an $k$-Borel pair $(B,T)$ of $G$. By Definition 
\ref{dfn:BKN}, the derived group $G_0=[G,G]$ is simply connected. 
Let $(B_0,T_0)=(B\cap G_0, T\cap G_0)$ be the corresponding $k$-Borel pair of $G_0$. Let $Z$ be the center of $G$ and $Z_{0}$ the center of $G_0$. Then we have the adjoint group $G_\ad=G/Z=G_0/Z_0$, and in particular, $T_\ad=T_0/Z_0=T/Z$. 

Associated with $(G_0,B_0,T_0)$, we take $\Delta$ to be the set of simple roots and
$\wh{\Delta}$ to be the set of the associated fundamental weights. 
For each fundamental weight $\ome\in\wh{\Delta}$, denote by $\rho_{\ome}$ the associated irreducible rational representation of $G_0$ with highest weight $\omega$. 
Denote by $V_\ome$ the underlying space of $\rho_\ome$.
Define $G_0^+:=(T_0\times G_0)/Z_0^{\Delta}$, where $Z_0^{\Delta}$ denotes the diagonal embedding of the center $Z_0$ of $G_0$ into 
$T_0\times G_0$. We may extend the representation $\rho_{\ome}$ of $G_0$ to $G_0^+$ by 
\begin{align}
  \rho_{\omega}^+(t,g)=\omega(w_0(t^{-1}))\rho_{\omega}(g), \quad {\rm for}\ (t,g)\in T_0\times G_0,
\end{align}
where $w_0$ is the longest Weyl element in the Weyl group $W=W(G_0,T_0)$. Each root $\alpha\in\Delta$ can also be extended to $G_0^+$ via $\alpha^+(t,g):=\alpha(t)$.

Following \cite{Vin95}, to construct the Vinberg universal monoid. Define a map $\iota$ by composing the following representation 
\begin{align}\label{iota-vum}
    (\prod_{\alpha\in\Delta}\alpha^+)\times(\prod_{\omega\in\widehat{\Delta}}\rho_{\omega}^+) \colon G_0^+\longrightarrow (\prod_{\alpha\in\Delta}\BG_m)\times(\prod_{\omega\in\widehat{\Delta}}\GL(V_\ome)) 
\end{align}
with the embedding
\[
    (\prod_{\alpha\in\Delta}\BG_m)\times(\prod_{\omega\in\widehat{\Delta}}\GL(V_\ome))
    \longrightarrow (\prod_{\alpha\in\Delta}\mathbb{A}^1)\times(\prod_{\omega\in\widehat{\Delta}}\End(V_\ome)).
\]
Take $\CM^+$ to be the normalization of the closure of the image $\iota(G_0^+)$ of $G_0^+$, which is the {\bf Vinberg universal monoid} 
associated with the given reductive group $G$. 
Denote by $G^+$ the unit group of the monoid $\CM^+$. It is clear that $G^+\times G^+$ acts on $\CM^+$ by
\[
m(g_1,g_2)=g_1^{-1}mg_2,\quad {\rm for}\ m\in\CM^+, (g_1,g_2)\in G^+\times G^+.
\]

Let $G^+_\der=[G^+,G^+]$ be the derived group of the unit group $G^+$. The abelianization morphism from $\CM^+$ to the GIT quotient 
$\CM^+\!\sslash\!(G^+_\der\times G^+_\der)$ by $G^+_\der\times G^+_\der$ is denoted by 
\begin{align}\label{abm}
    \pi^+\colon \CM^+\longrightarrow\CM^+\!\sslash\!(G^+_\der\times G^+_\der)=\prod_{\alp\in\Delta}\mathbb{A}^1.
\end{align}
It is a flat morphism as algebraic varieties with equi-dimensional reduced fibers (\cite{Vin95, Ngo, Ngo14}). 

Since the derived group $G_0$ of $G$ is split and simply connected, its dual group $G_0^{\vee}$ is of adjoint type. Let $G_{0}^{\vee,\rsc}$ be the simply connected cover of $G_0^{\vee}$. Then $G_{0}^{\vee,\rsc}$ is the dual group of the adjoint group $G_\ad$ 
of $G$. Let $T_0^\vee$ be the maximal torus of $G_0^\vee$ that is dual to $T_0$ and $T_0^{\vee,\rsc}$ be the maximal torus of $G_0^{\vee,\rsc}$ that is dual to the torus $T_\ad$. It is clear that $T_\ad=T_0/Z_0=T/Z$ and we have the following  exact sequences
\[
1\rightarrow Z_0\rightarrow T_0\rightarrow T_0/Z_0\rightarrow 1\quad {\rm and}\quad 
1\rightarrow Z_0^{\vee,\rsc}\rightarrow T_0^{\vee,\rsc}\rightarrow T_0^{\vee}\rightarrow 1,
\]
where $Z_0^{\vee,\rsc}$ is the center of $G_0^{\vee,\rsc}$ and is the Cartier dual of $Z_0$. 
Write $T^\vee$ for the maximal torus of $G^\vee$ that is dual to $T$, then we have that $X^*(T)=X_*(T^\vee)$ and $X^*(T^\vee)=X_*(T)$. 

Assume that the irreducible rational representation $\rho$ of the dual group $G^\vee$ in the given $\bkn$-triple 
$\Delta_{\bkn}(\Fd):=(G,G^\vee,\rho)_\Fd$ has a dominant weight $\lam\in X^*(T^\vee)$ as its highest weight. 
In this case, we may write $\rho=\rho_\lam$. 
This dominant weight $\lam$ can be viewed as a dominant co-character in $X_*(T)$ of $T$. By combining the natural quotient morphism $T\to T/Z=T_0/Z_0=T_{0, \ad}$, we regard such $\lam\in X_*(T)$ as an element in $X_*(T_{0, \ad})$,  which is still denoted by $\lam$ if it does not cause any confusion. Then we obtain its composition 
\begin{align}\label{lam-alp}
    \BG_m\longrightarrow T_0/Z_0\longrightarrow \prod_{\alpha\in\Delta}\BG_m,
\end{align}
which extends to $\mathbb{A}^1 \to \prod_{\alpha \in \Delta}\mathbb{A}^1$ and is still denoted by $\lambda$. The monoid $\CM_{\lambda}$ associated with $\lambda$ or the 
representation $\rho$ is defined to be the following fiber product:
\begin{align}\label{rhoM}
   \begin{array}[c]{ccc}
\CM_{\lambda}&\longrightarrow&\CM^{+}\\
\downarrow&&\downarrow\scriptstyle{\pi^+}\\
\mathbb{A}^1&\stackrel{\lambda}{\longrightarrow}&\prod_{\alpha \in \Delta}\mathbb{A}^1
\end{array}
\end{align}
where the abelianization morphism $\pi^+$ is as given in \eqref{abm}. Note the morphism $\lambda: \mathbb{A}^1 \to \prod_{\alpha \in \Delta}\mathbb{A}^1$ is normal and $M^+$ is normal, we have $M_\lambda$ is normal.

Let $G_{\lambda}$ be the unit group of $\CM_{\lambda}$, which fits into an exact sequence
\begin{align}\label{Glambda}
    1\longrightarrow G_0\longrightarrow G_{\lam}\stackrel{m_\lam}{\longrightarrow}\BG_m\longrightarrow 1
\end{align}
where $m_\lam$ is the abelianization morphism of $G_\lam$.
More precisely, by the definition of $G_0^+$ and by \eqref{iota-vum}, \eqref{lam-alp} and \eqref{rhoM}, the unit group $G_{\lambda}$ fits into the fiber product
\begin{align}\label{diag_G_lambda}
    \begin{array}[c]{ccc}
G_{\lambda}&\longrightarrow&G_0^+=(T_0\times G_0)/Z_0^{\Delta}\\
\downarrow&&\downarrow\scriptstyle{\pr_1}\\
\BG_m&\stackrel{\lambda}{\longrightarrow}& T_\ad=T_0/Z_0
\end{array}
\end{align}
and the multiplication of the group $G_{\lambda}$ is induced from these of $\BG_m$ and $G_0^+$. 

\begin{lem}\label{lem_G_lambda}
    We may write $G_{\lambda}=\BG_m\ltimes_{\lambda}G_0$, with the multiplication given by
\begin{align}\label{multGlam}
    (a_1,g_1)(a_2,g_2)=(a_1a_2,\lambda(a_2)^{-1}g_1\lam(a_2)g_2),\quad {\rm for}\ a_1,a_2\in\BG_m, g_1,g_2\in G_0.
\end{align}
Since $\lambda(a_1)\in T_0/Z_0$, the product $\lambda(a_2)g_1\lambda(a_2)^{-1}$ is well-defined. 
\end{lem}

\begin{proof}
    Note that we have an isomorphism of varieties 
    \[
    (T_0\times G_0)/Z_0^{\Delta}\cong (T_0/Z_0)\times G_0: [(t,g)]\mapsto ([t],t^{-1}g)
    \]
    with inverse given by $(T_0/Z_0)\times G_0\ni ([t],g)\mapsto [(t,tg)]\in (T_0\times G_0)/Z_0^{\Delta}$. Then the projection $\pr_1:(T_0\times G_0)/Z_0^{\Delta}$ transfers to the canonical projection $(T_0/Z_0)\times G_0\rightarrow T_0/Z_0$. Then $\BG_m\times G_0$ is the fiber product of $(T_0/Z_0)\times G_0$ with $\BG_m$ over $T_0/Z_0$ according to the transitivity of the product (see the last second paragraph in \cite[Page 89]{Har77}, 
    for instance): 
    %and using the notations there we have $X=G_0/T_0$, $S=\Spec\; \BC$, $S^{\prime}=T_0/Z_0$ and $S^{''}=\BG_m$):
\begin{align}\label{diag-1}
    \begin{array}[c]{ccc}
\BG_m\times G_0 &\stackrel{\lambda\times \Id}{\longrightarrow} & (T_0/Z_0)\times G_0\\
\downarrow&&\downarrow\scriptstyle{\pr_1}\\
\BG_m&\stackrel{\lambda}{\longrightarrow}& T_\ad=T_0/Z_0
\end{array}
\end{align}
with morphisms
\[
q_1^{\prime}=\lam\times\Id:\BG_m\times G_0\rightarrow (T_0/Z_0)\times G_0:(a,g)\mapsto (\lambda(a),g)
\]
and
\[
q_2=\pr_1:\BG_m\times G_0\rightarrow \BG_m:(a,g)\mapsto a.
\]
Moreover, for any variety $X$ with $f_1=(f_1^{(1)},f_1^{(2)}):X\rightarrow  (T_0/Z_0)\times G_0$ and $f_2:X\rightarrow \BG_m$ such that $\pr_1\circ f_1=\lambda\circ f_2$, there is a unique morphism $f:X\rightarrow \BG_m\times G_0$ such that $(\lambda\times\Id)\circ f=f_1$ and $q_2\circ f=f_2$, which is given by $x\mapsto (f_2(x),f_1^{(2)}(x))$.

Under the isomorphism $(T_0/Z_0)\times G_0\cong (T_0\times G_0)/Z_0^{\Delta}$, the morphism $q_1^{\prime}$ induces 
the corresponding morphism $q_1:\BG_m\times G_0\rightarrow (T_0\times G_0)/Z_0^{\Delta}$ is given by $(a,g)\mapsto [(\lambda(a),\lambda(a)g)]$. Then we obtain the following fiber product diagram:
\begin{align}\label{diag-2}
    \begin{array}[c]{ccc}
\BG_m\times G_0 &\stackrel{q_1}{\longrightarrow} & (T_0\times G_0)/Z_0^{\Delta}\\
\downarrow\scriptstyle{q_2}&&\downarrow\scriptstyle{\pr_1}\\
\BG_m&\stackrel{\lambda}{\longrightarrow}& T_\ad=T_0/Z_0
\end{array}
\end{align}
with $\BG_m\times G_0$ as the fiber product. Moreover, using the above identification, for any variety $X$ with $f_1=(f_1^{(1)},f_1^{(2)}):X\rightarrow  (T_0\times G_0)/Z_0^{\Delta}$ and $f_2:X\rightarrow \BG_m$ such that $\pr_1\circ f_1=\lambda\circ f_2$, there is a unique morphism $f:X\rightarrow \BG_m\times G_0$ such that $q_1\circ f=f_1$ and $q_2\circ f=f_2$ and is given by $x\mapsto (f_2(x),f_1^{(1)}(x)^{-1}\cdot f_1^{(2)}(x))$.  It remains to determine the multiplication structure explicitly. 

For given $(a_1,g_1),(a_2,g_2)\in \BG_m\times G_0$, we have that 
\[
q_2((a_1,g_1))\cdot q_2((a_2,g_2))=a_1a_2,
\]
and
\[
q_1((a_1,g_1))\cdot q_1((a_2,g_2))=[(\lambda(a_1),\lambda(a_1)g_1)]\cdot[(\lambda(a_2),\lambda(a_2)g_2)]=[(\lambda(a_1)\lambda(a_2),\lambda(a_1)g_1\lambda(a_2)g_2)],
\]
where the first multiplication is in $\BG_m$ and the second multiplication is in $(T_0\times G_0)/Z_0^{\Delta} $.
In this way, we get two morphisms
\[
(\BG_m\times G_0)\times (\BG_m\times G_0)\rightarrow \BG_m:((a_1,g_1),(a_2,g_2))\mapsto a_1a_2
\]
and
\[
(\BG_m\times G_0)\times (\BG_m\times G_0)\rightarrow (T_0\times G_0)/Z_0^{\Delta}:((a_1,g_1),(a_2,g_2))\mapsto [(\lambda(a_1)\lambda(a_2),\lambda(a_1)g_1\lambda(a_2)g_2)],
\]
whose corresponding morphism $(\BG_m\times G_0)\times (\BG_m\times G_0)\rightarrow (\BG_m\times G_0)$ is given by
\begin{align}\label{product-1}
    (a_1,g_1)\cdot(a_2,g_2)=(a_1a_2,\lambda(a_1a_2)^{-1}\lambda(a_1)g_1\lambda(a_2)g_2)=(a_1a_2,\lambda(a_2)^{-1}g_1\lambda(a_2)g_2)
\end{align}
according to the last paragraph.

As the fiber product of the fiber product diagram in \eqref{diag-2}, the group $\BG_m\times G_0$ with the 
multiplication given in \eqref{product-1} is isomorphic to the group $G_\lam$ because of the universality. 
\end{proof}

From the proof of Lemma \ref{lem_G_lambda}, it is easy to obtain that when the co-character $\lambda\colon \BG_m\rightarrow T_0/Z_0$ factors through $\BG_m\rightarrow T_0\rightarrow T_0/Z_0$, 
or equivalently, the representation $\rho_{\lambda}$ can factor through the adjoint group $G^{\vee}_\ad$, 
the unit group $G_\lam$ is isomorphic to the direct product $\BG_m\times G_0$, which is explicitly given by 
\begin{align}\label{eq:ad}
    G_{\lambda}\cong \BG_m\times G_0\quad {\rm with}\ (a,g)\mapsto (a,\lam(a)g).
\end{align}
For convenience, we state it as a corollary. 

\begin{cor}\label{cor:ad}
    If the representation $\rho_{\lambda}$ can factor through the adjoint group $(G^{\vee})_\ad$, then 
    the unit group $G_\lam$ is isomorphic to the direct product $\BG_m\times G_0$, which is explicitly given as in 
    \eqref{eq:ad}.
\end{cor}

For any $\bkn$-triple $\Delta_{\bkn}(\Fd):=(G,G^\vee,\rho)_\Fd$ with $\rho=\rho_\lam$, one checks easily that the natural homomorphism
\[
    G^\vee \to H_\rho^\vee\cong(\BG_m \times G^\vee) / \{ (\lambda(z)^{-1},z) \mid z \in Z^\vee \}, \quad g \mapsto [(1,g)]
\]
is an isomorphism. By Corollary \ref{cor:Hrhosc} and Lemma \ref{lem_G_lambda}, we obtain the following 

\begin{cor}\label{ug-VM}
For a given $\bkn$-triple $\Delta_{\bkn}(\Fd):=(G,G^\vee,\rho)_\Fd$ with $\rho=\rho_\lam$, the group $G$ is isomorphic 
to the unit group $G_\lam=G_\rho=G(\CM_\rho)$, where $\CM_\rho=\CM_\lam$ is the $L$-monoid as constructed via the Vinberg method in \eqref{rhoM}.
\end{cor}

\subsection{Construction via the Putcha-Renner method}\label{ssec-CPRM}
For a given $\bkn$-triple $\Delta_{\bkn}(\Fd):=(G,G^\vee,\rho)_\Fd$ with $\rho=\rho_\lam$ as before, we are going to recall the construction via the Putcha-Renner method the monoid $\CM^\rho$ and show that $G=G(\CM^\rho)$, the unit group of the monoid $\CM^\rho$. 

A series of papers by Putcha \cite{Put84, Put88}  and Renner \cite{Ren85a, Ren85b, Ren89, Ren90} classified the reductive monoid characterized by the toric embedding of maximal tori. The survey \cite{Ren05} is also a good reference. \cite{Ngo} gives an explicit construction over not necessarily algebraically closed fields.

Following the notations in \cite{Ngo}, we denote by $\Omega(\rho)$ the set of weights of $\rho$ and by $\xi(\rho)$ the convex cone generated by $\Omega(\rho)$. According to \cite[Proposition 5.1]{Ngo20}, since the central character of $\rho$ is the scalar multiplication, it follows that $\xi(\rho)$ is strictly convex. Let $\xi(\rho)^{\vee}$ be the sub-monoid of $X^*(T)$ consisting of $\alpha\in X^*(T)$ such that the restriction of $\alpha$ to $\Omega(\rho)$ takes non-negative values. Let $\alpha_1,\cdots,\alpha_k\in X^*(T)$ such that their orbits under $W$ generate $\xi(\rho)^{\vee}$. By replacing $\alpha_i$ by a $W$-conjugate if necessary, we may assume that each $\alpha_i$ lies in the positive Weyl chamber. Let $\omega_{\alpha_i}:G\rightarrow \GL(V_{d_i})$ be the representation of $G$ of highest weight $\alpha_i$, where $d_i$ is the dimension of the representation. Then the monoid $\CM^{\rho}$ is defined as the normalization of the closure of the image $G$ under
\[
\prod_{i=1}^k\omega_{\alpha_i} :G\rightarrow  \prod_{i=1}^k\End(V_{d_i}).
\]
According to \cite[Proposition 5.1]{Ngo}, the monoid $\CM^{\rho}$ is independent of the choice of the generators $\alpha_1,\cdots,\alpha_k$ and is characterized by the property that for every maximal torus $T$ of $G$, the normalization of its closure in $\CM^{\rho}$ corresponds to the strictly convex cone generated by the set of weights $\Omega(\rho)$ in the sense that the ring of regular functions on this toric embedding is $k[\xi(\rho)^{\vee}]$.

\subsection{Proof of Theorem \ref{thm:monoids}}\label{ssec-PThmM}
For a given $\bkn$-triple $\Delta_{\bkn}(\Fd):=(G,G^\vee,\rho)_\Fd$ with $\rho=\rho_\lam$ as before, we constructed 
in Section \ref{ssec-CVM} a monoid $\CM_\rho$ with $G=G(\CM_\rho)$, the unit group of $\CM_\rho$, and in Section 
\ref{ssec-CPRM} a monoid $\CM^\rho$ with $G=G(\CM^\rho)$, the unit group of $\CM^\rho$. To prove Theorem \ref{thm:monoids}, it is enough to prove that $\CM_\rho\cong\CM^\rho$. 

We regard $\lambda$ as an dominant element in $X_*(T_{0, \ad})=X^*(T_0^{\vee, \rsc })$ as in Section \ref{ssec-CVM}. Recall from Corollary \ref{ug-VM} that $G\cong G_{\lambda}=\BG_m\ltimes_{\lambda}G_0$. If $\lambda=0$, then it's easy to check the result. In the following, we assume $\lambda \neq 0$. 

Take the maximal spit torus $T_\lambda= \BG_m \times T_0$ then recall the following commutative diagram
% https://q.uiver.app/#q=WzAsOCxbMSwwLCJcXGJ1bGxldCJdLFsyLDAsIlxcYnVsbGV0Il0sWzMsMCwiXFxidWxsZXQiXSxbMCwwLCJBIl0sWzAsMSwiXFxidWxsZXQiXSxbMSwxLCJcXGJ1bGxldCJdLFsyLDEsIlxcYnVsbGV0Il0sWzMsMSwiXFxidWxsZXQiXSxbMCwxXSxbMSwyXSxbMywwXSxbMyw0XSxbNCw1XSxbNSw2XSxbNiw3LCJcXHBpXisiLDJdLFsyLDcsIlxcbGFtYmRhIl0sWzAsNV0sWzEsNl1d
\[\begin{tikzcd}
	T_\lambda=\BG_m \times T_0 & G_{\lambda}=\BG_m\ltimes_{\lambda}G_0 & M_\rho & \mathbb{A}^1 \\
	T_0^+=(T_0\times T_0)/Z_0^{\Delta} & G_0^+ & M^+ & \prod_{\alpha \in \Delta}\mathbb{A}^1
	\arrow[from=1-1, to=1-2]
	\arrow[from=1-1, to=2-1]
	\arrow[from=1-2, to=1-3]
	\arrow[from=1-2, to=2-2]
	\arrow[from=1-3, to=1-4]
	\arrow[from=1-3, to=2-3]
	\arrow["\lambda", from=1-4, to=2-4]
	\arrow[from=2-1, to=2-2]
	\arrow[from=2-2, to=2-3]
	\arrow["{\pi^+}"', from=2-3, to=2-4]
\end{tikzcd}\]
where every square is Cartesian and the left vertical arrow is given by
\[
    \BG_m \times T_0 \to (T_0\times G_0)/Z_0^{\Delta}, \quad (a, t) \mapsto [\lambda(a), \lambda(a)t].
\]
Recall $\CM^+$ is the normalization of the image of the following map
\begin{align*}
    \iota: G_0^+= (T_0\times G_0)/Z_0^{\Delta} &\longrightarrow (\prod_{\alpha \in \Delta}\mathbb{A}^1) \times (\prod_{\ome \in \widehat{\Delta}}\End(V_\ome)), \\
    [t, g] &\longmapsto (\prod_{\alpha \in \Delta}\alpha(t)) \times (\prod_{\ome \in \widehat{\Delta}}\ome(w_0(t^{-1}))\rho_\ome(g)) .
\end{align*}
Denote by $\xi^+ \subset X^*(T_0^+) \otimes \mathbb{R}$ the convex cone generated by all of the weights in $\iota$. Let $\overline{T}_0^+$ denote the Zariski closure of $T_0^+$ in  $M^+$. According to \cite[Theorem 5.4]{Ren05}, $\overline{T}_0^+$ is a normal affine toric varieties. It follows from \cite[Section 3.3]{Ren05} that $\overline{T}_0^+=\Spec\; k[X(\overline{T}_0^+)]$, where the character lattice of $\overline{T}_0^+$ is given by
\[
    X(\overline{T}_0^+)=\xi^+ \cap X^*(T_0^+).
\]
Let $\overline{T}_\lambda$ denote the Zariski closure of $T_\lambda$ in $\CM_\rho$. Note $\lambda \neq 0$, then the commutative diagram
\[\begin{tikzcd}
	\overline{T}_\lambda & \mathbb{A}^1 \\
	\overline{T}_0^+ & \prod_{\alpha \in \Delta}\mathbb{A}^1
	\arrow[from=1-1, to=1-2]
	\arrow[from=1-1, to=2-1]
	\arrow["\lambda", from=1-2, to=2-2]
	\arrow[from=2-1, to=2-2]
\end{tikzcd}\]
is Cartesian.

Identify the character lattice $X^*(T_\lambda)$ of $T_\lambda=\BG_m \times T_0$ with $\mathbb{Z}\oplus X^*(T_0)$. Denote by $\xi_\lambda \subset X^*(T_\lambda)\otimes_{\mathbb{Z}}\mathbb{R}$ the convex cone generated by the following characters
\[
    (1, 0) \text{ and } (\langle \ome^\prime-w_0 \ome, \lambda \rangle, \ome^\prime), \ome \in \widehat{\Delta}.
\]
Here $\ome^\prime$ denotes any weight in $\rho_\ome$ and $\langle \  , \  \rangle$ denotes the pairing between $X^*(T_{0, \ad})$ and $X_*(T_{0, \ad})$. Note $\ome^\prime-w_0 \ome \in X^*(T_{0, \ad})$, so the pairing is well-defined.

By the Cartesian diagram above, one calculates directly that $\overline{T}_\lambda=\Spec\; k[X(\overline{T}_\lambda)]$, where the character lattice of $\overline{T}_\lambda$ is given by
\[
    X(\overline{T}_\lambda)=\xi_\lambda \cap X^*(T_\lambda).
\]

On the other hand, by the same arguments as in Proposition \ref{prp:es}, we have the Langlands dual group of $G \cong G_\lambda$ is
\[
    G^\vee_\lambda= \BG_m \times_{(\lambda^{-1}, \Id)(Z_0^{\vee, \rsc})}G_0^{\vee, \rsc} \cong G^\vee.
\]
The isomorphism is induced from $\Fd^\vee\colon \BG_m\to G^\vee$ and the homomorphism $G_0^{\vee, \rsc} \to G^\vee$, which is dual to the quotient $G \to G_\ad=G_{0, \ad}$. Under this isomorphism, we may write the representation $\rho$ of $G^\vee$ as the representation of $G^\vee_\lambda$ given by
\[
    G^\vee_\lambda= \BG_m \times_{(\lambda^{-1}, \Id)(Z_0^{\vee, \rsc})}G_0^{\vee, \rsc} \to \GL(V_\rho), \quad (a, g) \mapsto a \I_{V_\rho} \cdot  \rho^\prime_\lambda(g),
\]
where $\rho^\prime_\lambda$ denotes the highest weight representation of $G_0^{\vee, \rsc}$ associated with $\lambda \in X^*(T_0^{\vee, \rsc })$.

Take the maximal torus $\BG_m \times_{(\lambda^{-1}, \Id)(Z_0^{\vee, \rsc})}T_0^{\vee, \rsc}$ which has the following isomorphism
\begin{align*}
    \BG_m \times_{(\lambda^{-1}, \Id)(Z_0^{\vee, \rsc})}T_0^{\vee, \rsc} &\cong \BG_m \times (T_0^{\vee, \rsc})_{\ad}= \BG_m \times T_0^\vee =T_\lambda^\vee \\
    [a, t] &\mapsto (\lambda(t)a, [t]).
\end{align*} 
Identify the character lattice $X^*(\BG_m \times_{(\lambda^{-1}, \Id)(Z_0^{\vee, \rsc})}T_0^{\vee, \rsc})$ with $\mathbb{Z} \oplus X^*(T_0^\vee)$. Then the weights of $\rho_\lambda$ are given by
\[
    (1, \lambda^\prime- \lambda)
\]
where $\lambda^\prime \in X^*(T_0^{\vee, \rsc})$ denotes any weight $\rho^\prime_\lambda$. Note $\lambda^\prime- \lambda \in X^*(T_0^\vee)$. So
$\xi(\rho) \subset X_*(T_\lambda)\otimes_{\mathbb{Z}} \mathbb{R}$ is equal to the convex cone generated by the weights above. Let $\xi(\lambda)^\vee$ be its dual cone. 

Let $\overline{T}^\lambda$ denote the Zariski closure of $T_\lambda$ in $\CM^\rho$. Similarly $\overline{T}^\lambda$ is the normal affine toric variety $\Spec \;k[X(\overline{T}^\lambda)]$, where
\[
    X(\overline{T}^\lambda)=\xi(\lambda)^\vee \cap X^*(T_\lambda)=\xi(\rho)^\vee.
\]
In the following, we prove that $\xi_\lambda=\xi(\lambda)^\vee$, which implies that $\overline{T}_\lambda \cong \overline{T}^\lambda$. Then by \cite[Theorem 5.4]{Ren05} this isomorphism extends to an isomorphism $\CM_\rho \cong \CM^\rho$.

\begin{prp}
    With the notations introduced above, then $\xi_\lambda=\xi(\lambda)^\vee$ holds.
\end{prp}
For the proof, we notice the isogeny $T_0 \to T_{0, \ad}$ induces isomorphisms
\[
    X^*(T_{0, \ad})\otimes_\mathbb{Z} \mathbb{R} \cong X^*(T_0)\otimes_\mathbb{Z} \mathbb{R}, \quad X_*(T_0)\otimes_\mathbb{Z} \mathbb{R} \cong X_*(T_{0, \ad})\otimes_\mathbb{Z} \mathbb{R}.
\]
So we can use the same notation $\langle \  , \  \rangle$ to denote the pairing between $X^*(T_{0, \ad})\otimes_\mathbb{Z} \mathbb{R} \cong X^*(T_0)\otimes_\mathbb{Z} \mathbb{R}$ and $X_*(T_0)\otimes_\mathbb{Z} \mathbb{R} \cong X_*(T_{0, \ad})\otimes_\mathbb{Z} \mathbb{R}$. If the context is clear, we may also use $\langle \  , \  \rangle$ to denote the pairing between 
\[
X^*(T_\lambda)\otimes_{\mathbb{Z}}\mathbb{R}=\mathbb{R}\oplus X^*(T_0)\otimes_\mathbb{Z} \mathbb{R}
\quad {\rm and}\quad 
X_*(T_\lambda)\otimes_{\mathbb{Z}}\mathbb{R}=\mathbb{R}\oplus X_*(T_0)\otimes_\mathbb{Z} \mathbb{R}.
\]

\begin{proof}
    We first prove that $\xi_\lambda \subset \xi(\lambda)^\vee$. By definition, we have that 
    \[
        \langle (1, 0), (1, \lambda^\prime-\lambda) \rangle=1 > 0,
    \]
    and
    \begin{align*}
        \langle (\langle \ome^\prime-w_0 \ome, \lambda \rangle, \ome^\prime), (1, \lambda^\prime- \lambda) \rangle 
        =\langle \ome^\prime-w_0 \ome, \lambda \rangle + \langle \ome^\prime, \lambda^\prime-\lambda \rangle= \langle \ome^\prime, \lambda^\prime \rangle- \langle w_0 \ome, \lambda \rangle.
    \end{align*}  
    Note that any weight in the highest weight representation associated with $\lambda$ lies in the convex hull of $\{ W \lambda \}$. It follows that for a fixed $\ome^\prime$, one has
    \[
        \langle \ome^\prime, \lambda^\prime \rangle \geq \langle \ome^\prime, w \lambda\rangle=\langle w^{-1}\ome^\prime, \lambda \rangle
    \]
    for certain $w \in W$, which implies that 
    \[
        \langle (\langle \ome^\prime-w_0 \ome, \lambda \rangle, \ome^\prime), (1, \lambda^\prime- \lambda) \rangle \geq \langle w_0(w_0 w^{-1}\ome^\prime-\ome), \lambda \rangle.
    \]
    Since $\ome$ is the highest weight in $\rho_\ome$, we have
    \[  
        \ome - w_0 w^{-1}\ome^\prime =\sum_{\alpha \in \Delta} n_\alpha \alpha, \quad n_\alpha \geq 0
    \]
    and 
    \[
        w_0(w_0 w^{-1}\ome^\prime-\ome)=\sum_{\alpha \in \Delta} n_\alpha (-w_0\alpha).
    \]
    Since $-w_0 \Delta =\Delta$ and $\lambda$ is dominant we obtain that $\langle w_0(w_0 w^{-1}\ome^\prime-\ome), \lambda \rangle \geq 0$, which implies that $\xi_\lambda \subset \xi(\lambda)^\vee$.

    It remains to prove that $\xi(\lambda)^\vee \subset \xi_\lambda$. Suppose that $(a, \mu^\prime) \in \xi(\lambda)^\vee$. Then $\mu^\prime=w \mu$ for certain $w \in W$ and dominant $\mu \in X_*(T_0)\otimes_\mathbb{Z} \mathbb{R}$. It follows that 
    \[
        \langle (a, \mu^\prime), (1, \lambda^\prime-\lambda)\rangle=a+ \langle w \mu, \lambda^\prime-\lambda \rangle= a+ \langle  \mu, w^{-1} \lambda^\prime \rangle- \langle w\mu, \lambda \rangle \geq 0.
    \]
    By taking that $\lambda^\prime=w w_0 \lambda$, we obtain that 
    \[
        a \geq \langle w\mu, \lambda \rangle- \langle  \mu, w_0 \lambda \rangle= \langle w \mu-w_0 \mu, \lambda \rangle.
    \]
    Since $\mu$ is dominant, we write that $\mu =\sum_{\ome \in \widehat{\Delta}}m_\ome \ome$ for $m_\ome \geq 0$. 
    This implies 
    \[
        (a, \mu^\prime)=(a-\langle w \mu-w_0 \mu, \lambda \rangle)(1, 0)+ \sum_{\ome \in \widehat{\Delta}}m_\ome (\langle w\ome-w_0 \ome, \lambda \rangle, w\ome) \in \xi_\lambda.
    \]

\end{proof}

%%%%%%%%%%%%%%%
\section{Examples}\label{sec-Ex}
%%%%%%%%%%%%%%%

We provide the explicit $\bkn$-triples $\Delta_\fp=(H_\rho,H_\rho^\vee,\rho_\fp)_{\Fd_\fp}=\Delta_\bkn(\Fd_\fp)$ for some examples of $L$-triples $(G,G^\vee,\rho)$ in order to illustrate the discussions above. 

\subsection{Symmetric powers of $\GL_2$}\label{ssec-symnGL2}
We take $(G,G^\vee,\rho)$ to be $(\GL_2,\GL_2,\sym^n)$. Since $\rho=\sym^n$, the symmetric $n$-th power of $\GL_2$, we must have that the highest weight $\lam_n=n\ome_1$, where $\ome_1$ is the first fundamental weight of $\GL_2(\BC)$.

In the case $G=\GL_2$, we have that $G_0=\SL_2$. We take $T$ to be the maximal torus of $\GL_2$ consisting of all diagonal matrices, then $T_0=T\cap\SL_2$ consisting of all diagonal matrices in $\SL_2$, and then $Z_0=\{\pm\RI_2\}\subset T_0$. Since the co-character is given by 
    \[
    \lambda=\lambda_n\colon \BG_m\rightarrow T_0/Z_0\cong\BG_m, \quad x\mapsto x^n,
    \]
We have that $H_\rho=\BG_m\times_\lam\SL_2$. We conclude that $H_\rho\cong \BG_m\times \SL_2$ when $n$ is even and $H_\rho\cong\GL_2$ when $n$ is odd, 
which can be verified as follows. 

When $n$ is even, since $\lambda$ could lift to a co-character 
\[
\lam\colon \BG_m\rightarrow T_0\colon 
x\mapsto \begin{pmatrix}
    x^{n/2} & 0 \\ 0 & x^{-n/2}
\end{pmatrix}, 
\]
by Corollary \ref{cor:ad}, which is a special case of Lemma \ref{lem_G_lambda}, we have that $G_{\lambda}\cong\BG_m\times\SL_2$.

When $n$ is odd, there is an algebraic group isomorphism 
    \[
    \BG_m\ltimes_{\lambda}\SL_2\rightarrow \GL_2\colon (a,g)\mapsto \begin{pmatrix}
        a^{\frac{n+1}{2}} & 0\\0 & a^{\frac{1-n}{2}}
    \end{pmatrix}g
    \]
with inverse given by 
\[
g\mapsto \left( \det g, \begin{pmatrix}
     \det ^{\frac{-n-1}{2}} g& 0\\ 0& \det^{\frac{n-1}{2}} g
\end{pmatrix} g      \right).
\]

On the dual side, by Corollary \ref{cor:Hrhod}, we have that 
\[
G_{\lambda}^{\vee}=\BG_m\times\SL_2/\{(\pm 1)^n,\pm\RI_2\}. 
\]
When $n$ is even, we have that $(\pm1)^n=1$ and hence $G_\lam^\vee$ is isomorphic to $\BG_m\times\PGL_2$.
When $n$ is odd, we have the following isomorphism of algebraic groups:
\begin{align}\label{iso of dual of gl_2}
G_{\lambda}^{\vee}\cong \GL_2\colon (a,g)\mapsto a\RI_2\cdot g.
\end{align}

Finally, we consider the representation $\rho_\fp$. 
When $n$ is even, the representation is given by 
\begin{align}\label{rhoFne}
    \rho_\fp=\st_{\BG_m}\otimes \sym_n^*,
\end{align}
where $\st_{\BG_m}$ is the standard representation of $\BG_m$ acting by scalar, and $\sym_n^*$ is the induced symmetric $n$-th power on $\PGL_2$. When $n$ is odd, according to the isomorphism (\ref{iso of dual of gl_2}), the representation $\rho_\fp$ is given as follows. For any $g\in\GL_2$, choose $a \in\BG_m$ such that $a^2=\det g$, which implies that $a^{-1}g\in\SL_2$. 
Hence we obtain that 
\begin{align}\label{rhoFno}
    \rho_\fp(g)=a \cdot \sym^n(a^{-1}g).
\end{align}
Note that this is not the symmetric $n$-th power of $\GL_2$. We summarize the discussion as a proposition.

\begin{prp}\label{prp:GL2symn}
    For any integer $n\geq 1$, let $\lam=n\ome_1$ be the dominant weight for the symmetric $n$-th power representation 
    $\rho_\lam$ of $\GL_2$. Then the $\bkn$-triple $\Delta_\bkn=(H_\rho,H_\rho^\vee,\rho_\fp)_{\Fd_\fp}$ associated to the given $L$-triple $(\GL_2,\GL_2,\rho_\lam)$ is explicitly given as follows. 
\begin{enumerate}
    \item The group $H_\rho$ is given by 
    \[
    H_\rho=\BG_m\ltimes_{\lambda}\SL_2\cong
    \begin{cases}
        \BG_m\times\SL_2& {\rm if}\ n\ {\rm is\ even};\\
        \GL_2& {\rm if}\ n\ {\rm is\ odd}.
    \end{cases}
    \]
    \item The dual group $H_\rho^\vee$ is given by 
    \[
    H_\rho^\vee\cong
    \begin{cases}
        \BG_m\times\PGL_2& {\rm if}\ n\ {\rm is\ even};\\
        \GL_2& {\rm if}\ n\ {\rm is\ odd}.
    \end{cases}
    \]
    \item The representation $\rho_\fp$ is given by 
    \[
    \rho_\fp=
    \begin{cases}
        \st_{\BG_m}\otimes \sym_n^*\ {\rm as\ in\ \eqref{rhoFne}}, & {\rm if}\ n\ {\rm is\ even};\\
        \rho_\fp(g)=a \cdot \sym^n(a^{-1}g)\ {\rm as\ in\ \eqref{rhoFno}},& {\rm if}\ n\ {\rm is\ odd}.
    \end{cases}
    \]
    \item The generalized determinant $\Fd_\fp\colon H_\rho\longrightarrow\BG_m$ is given by 
 \[
  H_{\rho}\cong
    \begin{cases}
        \BG_m\times\SL_2: (a,g)\mapsto a , & {\rm if}\ n\ {\rm is\ even};\\
        \GL_2: g\mapsto \det g   ,& {\rm if}\ n\ {\rm is\ odd}.
    \end{cases}
    \]
\end{enumerate}
\end{prp}
We also refer to \cite{Sha17, SS22, Luo24} for discussions of this case.

\subsection{The standard representation of $\GL_n$}\label{ssec-StdGLn}
    In this case, the given triple is $(\GL_n,\GL_n,\rho_\lam)$ with $\rho_{\lambda}$ being the standard representation of $\GL_n$ associated to the first fundamental weight $\ome_1$. 
    To explicate the associated $\bkn$-triple $\Delta_\bkn=(H_\rho,H_\rho^\vee,\rho_\fp)_{\Fd_\fp}$, we take $T$ to 
    be the maximal torus of $\GL_n$ consisting of all diagonal matrices, and $Z$ be the center of $\GL_n$ consisting of all scalar matrices. Then $T_0=T\cap\SL_n$ and $Z_0=Z\cap\SL_n=\mu_n$, where $\mu_n$ is the group of the $n$-th roots of unity.
Consider the following isomorphism:
    \[
    T_0/Z_0\cong T/Z\cong\BG_m^{n-1}\colon \diag(t_1,\cdots,t_n)\mapsto \left( \frac{t_1}{t_n},\cdots,\frac{t_{n-1}}{t_n}   \right).
    \]
The highest weight $\lambda=\ome_1$ of the standard representation defines a co-character: 
    \[
    \lambda\colon \BG_m\longrightarrow T/Z\cong \BG_m^{n-1},\quad t\mapsto \diag(t,1,\cdots,1)\mapsto (t,1,\cdots,1).
    \]
Then $H_\rho=\BG_m\ltimes_{\lambda}\SL_n$ with group law given by 
\begin{align*}
(\BG_m\ltimes_{\lambda}\SL_n)\times(\BG_m\ltimes_{\lambda}\SL_n)&\rightarrow\BG_m\ltimes_{\lambda}\SL_n:\\((a,g),(a^{\prime},g^{\prime}))&\mapsto (aa^{\prime}, \diag((a^{\prime})^{-1},1,\cdots,1)g\cdot \diag(a^{\prime},1,\cdots,1)\cdot g^{\prime}).
\end{align*}
As algebraic groups, we have the following isomorphism:
\[
\BG_m\ltimes_{\lambda}\SL_n\cong\GL_n\colon (a,g)\mapsto \diag(a,1,\cdots,1)\cdot g,
\]
with the inverse given by
\[
\GL_n\ni g\mapsto(\det g, \diag(\det{^{-1} g},1,\cdots,1)\cdot g)\in\BG_m\ltimes_{\lambda}\SL_n.
\]
On the dual side, we have that 
\begin{align*}
    G_{\lambda}^{\vee}
    &=\BG_m\times\SL_n/\{(\omega_{\lambda}(z)^{-1},z)\mid z\in\mu_n\subset\SL_n\}\\
    &=G_{\lambda}^{\vee}\\
    &=\BG_m\times\SL_n/\{(z^{-1},z)\mid z\in\mu_n\subset\SL_n\},
\end{align*}
which is isomorphic to $\GL_n$ via the map 
\[
G_{\lambda}^{\vee}\rightarrow\GL_n:[(a,g)]\mapsto ag.
\]
Finally according to the construction of $\rho_\fp$ in this case, it is still the standard representation of $\GL_n$.

To summarize the discussion above, we have
\begin{prp}\label{prp:stdn}
    Let $\lambda$ be the highest weight for the standard representation $\rho_{\lambda}$ of $\GL_n$. Then 
    the $\bkn$-triple $\Delta_\bkn=(H_\rho,H_\rho^\vee,\rho_\fp)_{\Fd_\fp}$ associated to the $L$-triple $(\GL_n,\GL_n,\rho_{\lambda})$ is explicitly given as follows.
    \begin{enumerate}
    \item The group $H_\rho$ is given by 
    \[
    H_\rho=\BG_m\ltimes_{\lambda}\SL_n\cong\GL_n
    \]
    \item The dual group $H_\rho^\vee$ is given by 
    \[
    H_\rho^\vee\cong\GL_n
    \]
    \item The representation $\rho_\fp$ is given by 
    \[
    \rho_\fp=\rho_{\lambda}=\st_{\GL_n},
    \]
    the standard representation of $\GL_n$.
    \item The generalized determinant $\Fd_\fp\colon H_\rho\longrightarrow\BG_m$ is given by 
    \[
        H_\rho \cong \GL_n \to \BG_m : g \mapsto \det(g).
    \]
\end{enumerate}
\end{prp}

\subsection{The adjoint representation of $\GL_n$}\label{ssec-AdGLn}
Let $G=\GL_n$ and $\rho_{\lambda}$ be the adjoint representation of $\GL_n(\BC)$ and $\lambda$ be its highest weight. In this case, we have that $G_0=\SL_n$. Take $T$ and $T_0$ as in Section \ref{ssec-StdGLn}. 

Since the character $\lambda$ can be lifted to a co-character $\BG_m\rightarrow T_0:a\mapsto \diag(a,1,\cdots,1,a^{-1})$,
by Corollary \ref{cor:ad}, which is a special case of Lemma \ref{lem_G_lambda}, we obtain that $H_\rho=\BG_m\times \SL_n$. 

On the dual side, we have 
\[
H_\rho^{\vee}=\BG_m\times\SL_n/\{(\omega_{\lambda}(z))^{-1},z)\mid z\in \mu_n\subset\SL_n\}
\cong\BG_m\times(\SL_n/\mu_n),
\]
where $\mu_n$ is the finite group scheme of $n$-th roots of unity, and the last equality is because $\omega_{\lambda}(z)=1$ for all $z\in\mu_n$. As algebraic groups, we obtain that $G_{\lambda}^{\vee}=\BG_m\times \PGL_n$. 

Finally the corresponding representation is $\rho_\fp=\st_{\BG_m}\otimes\Ad$, where $\Ad$ is the adjoint representation of $\PGL_n$ induced from that of $\SL_n$. To summarize, we have the following proposition.

\begin{prp}\label{prp:Adn}
    Let $\lambda$ be the highest weight for the adjoint representation $\rho_{\lambda}$ of $\GL_n$. 
    Then the $\bkn$-triple $\Delta_\bkn=(H_\rho,H_\rho^\vee,\rho_\fp)_{\Fd_\fp}$ associated to the triple $(\GL_n,\GL_n,\rho_{\lambda})$ is explicitly given as follows.
    \begin{enumerate}
    \item The group $H_\rho$ is given by 
    \[
    H_\rho=\BG_m\ltimes_{\lambda}\SL_n\cong\BG_m\times\SL_n
    \]
    \item The dual group $H_\rho^\vee$ is given by 
    \[
    H_\rho^\vee\cong
    \BG_m\times\PGL_n
    \]
    \item The representation $\rho_\fp$ is given by 
    \[
    \rho_\fp=\st_{\BG_m}\otimes\Ad,
    \]
    where $\Ad$ is the adjoint representation of $\PGL_n$ induced from that of $\SL_n$.
    \item The generalized determinant $\Fd_\fp\colon H_\rho\longrightarrow\BG_m$ is given by 
    \[
    \BG_m\times\SL_n\rightarrow\BG_m:(a,g)\mapsto a
    \]
\end{enumerate}
\end{prp}

\subsection{The symmetric square representation of $\GL_n$}\label{ssec-sym2GLn}
In this case, we consider the triple $(\GL_n,\GL_n,\rho_\lam)$ with $G=\GL_n$ and $\rho_\lambda$ being the symmetric square representation $\GL_n \to \GL_{\frac{n(n+1)}{2}}$, associated to the highest weight $2\ome_1$. 
Take $T$ and $T_0$ as in Section \ref{ssec-StdGLn}.

The highest weight $\lambda=2\ome_1 \in X^*(T^\vee)=X_*(T)$ of the symmetric square representation of $G^\vee$ gives 
a co-character:
\[
    \lambda\colon \BG_m\longrightarrow T,\quad a\mapsto \diag(a^2,1,\cdots,1).
    \]
Similar to the previous case, we can write $H_\rho$ as the image of the following homomorphism
\[
    \BG_m\ltimes_{\lambda}\SL_n\longrightarrow \BG_m \times \GL_n : (a,g) \mapsto (a,\diag(a^2,1,\cdots,1)\cdot g)
\]
which is clearly equal to 
\[
    \{ (a,g) \in \BG_m \times \GL_n \mid \det(g)=a^2 \}.
\]
More specifically, when $n=2l$, we have the following isomorphism
\begin{align*}
    \BG_m \times \SL_n / \{(z^{-1}, z\RI_n \mid z\in \mu_l  \} \cong H_\rho, \quad 
    [(a, g)] \mapsto (a^l,g \cdot a \RI_n),
\end{align*}
which is the same as in \cite{Sha97}.

Note on the dual side, we have that 
\[
    H_\rho^{\vee}=\BG_m\times\SL_n/\{(z^{-2},z\RI_n)\mid z\in\mu_n\},
\]
with the following isomorphisms
\begin{align*}
    \BG_m\times\SL_n/\{(z^{-2},z\RI_n)\mid z\in\mu_n\}&\cong\BG_m\times\GL_n/\{(z^{-2},z\RI_n)\mid z\in\BG_m \} : [(a,g)] \mapsto [(a,g)],\\
    \GL_n/\{\pm \RI_n \} &\cong \BG_m\times\GL_n/\{(z^{-2},z\RI_n)\mid z\in\BG_m \} : [g] \mapsto [(1,g)].
\end{align*}
Then the unique morphism $\GL_n \to H_\rho^{\vee}$ is given by the central isogeny
\[
    \GL_n \longrightarrow \GL_n/\{\pm\RI_n \}: g \mapsto [g].
\]
Since the symmetric square representation has its kernel equal to $\{ \pm \RI_n \}$, the given representation $\rho_\lambda$ descents to a representation $\rho^{\Fn}$ of $\GL_n/\{\pm\RI_n \}$.

Dual to the morphism $G^\vee \to G_{\lambda}^{\vee}$, we have the morphism $G_\lambda \to G$, which is explicitly given by 
\[
    \{ (a,g) \in \BG_m \times \GL_n \mid \det(g)=a^2 \} \to \GL_n : (a,g) \mapsto g.
\]
To summarize the discussion above, we have
\begin{prp}
    Let $\lambda=2\ome_1$ be the highest weight for the symmetric square representation $\rho_{\lambda}$ of $\GL_n$. Then the $\bkn$-triple $\Delta_\bkn=(H_\rho,H_\rho^\vee,\rho_\fp)_{\Fd_\fp}$ associated to the given triple $(\GL_n,\GL_n,\rho_{\lambda})$ is explicitly given as follows.
    \begin{enumerate}
    \item The group $H_\rho$ is given by 
    \[
    H_\rho=\BG_m\ltimes_{\lambda}\SL_n\cong\ \{ (a,g) \in \BG_m \times \GL_n \mid \det(g)=a^2 \}
    \]
    \item The dual group $H_\rho^\vee$ is given by 
    \[
    H_\rho^\vee\cong \GL_n/\{\pm\RI_n \}
    \]
    \item The representation $\rho_\fp$ is given by the descending of $\rho_\lambda$ from $\GL_n$ to $H_\rho^\vee$.
    \item The generalized determinant $\Fd_\fp\colon H_\rho\longrightarrow\BG_m$ is given by
    \[
        H_\rho \cong \{ (a,g) \in \BG_m \times \GL_n \mid \det(g)=a^2 \} \to \BG_m :  (a, g) \mapsto a.
    \]
\end{enumerate}
\end{prp}

\subsection{The doubling case of $\Sp_{2n}$}\label{ssec-DBSP} 
In this case, the given triple is $(\Sp_{2n},\SO_{2n+1},\rho_\lam)$ with $G=\Sp_{2n}=G_0$. Let $T=T_0$ be a maximal torus of $G$. Let $\rho_{\lambda}$ be the standard representation of $\SO_{2n+1}$ with $\lambda$ its highest weight. 
Since $T=T_0$ and $\lambda$ is already a co-character of $T_0$, by Corollary \ref{cor:ad}, which is a special case of Lemma \ref{lem_G_lambda}, we have in this case that $H_\rho=\BG_m\times\Sp_{2n}$.
On the dual side, we have that 
\[
H_\rho^{\vee}=\BG_m\times \mathrm{Spin}_{2n+1}/\{(1,\pm \RI_{2n+1})\}\cong\BG_m\times\SO_{2n+1}.
\]
Finally due to the construction of $\rho_\fp$, we see in this case that $\rho_\fp=\st_{\BG_m}\otimes\st_{\SO_{2n+1}}$, where $\st_{\BG_m}$ is the standard representation of $\BG_m$ and $\st_{\SO_{2n+1}}$ is the standard presentation of $\SO_{2n+1}$. In summary, we have the following. 

\begin{prp}\label{prp:StdSp}
    Let $\lambda$ be the highest weight for the standard representation $\rho_{\lambda}$ of $\SO_{2n+1}$. Then 
    the $\bkn$-triple $\Delta_\bkn=(H_\rho,H_\rho^\vee,\rho_\fp)_{\Fd_\fp}$ associated to the triple $(\Sp_{2n},\SO_{2n+1},\rho_{\lambda})$ is explicitly given as follows.
    \begin{enumerate}
    \item The group $H_\rho$ is given by 
    \[
    H_\rho=\BG_m\ltimes_{\lambda}\Sp_{2n}\cong\BG_m\times\Sp_{2n}
    \]
    \item The dual group $H_\rho^\vee$ is given by 
    \[
    H_\rho^\vee\cong
    \BG_m\times\SO_{2n+1}
    \]
    \item The representation $\rho_\fp$ is given by 
    \[
    \rho_\fp=\st_{\BG_m}\otimes\st_{\SO_{2n+1}},
    \]
    where $\st_{\BG_m}$ is the standard representation of $\BG_m$ and $\st_{\SO_{2n+1}}$ is the standard presentation of $\SO_{2n+1}$.
    \item The generalized determinant $\Fd_\fp\colon H_\rho\longrightarrow\BG_m$ is given by
    \[
        H_\rho\cong\BG_m\times\Sp_{2n} \to \BG_m : (a, g) \mapsto a.
    \]
\end{enumerate}
\end{prp}

%%%%%%%%%%%%%%%%
\section{Borel's Conjecture}\label{sec-BC}
%%%%%%%%%%%%%%%

In order to provide more explicit information of the representation $\sig$ for the given $\pi$ in Theorem \ref{thm:localm}, we intend to show that the morphism 
\begin{align}\label{etafp}
    \eta_\fp\colon H_\rho\longrightarrow G
\end{align}
which is dual to the morphism $\eta_\fp^\vee\colon G^\vee\longrightarrow H_\rho^\vee$ as defined in \eqref{diag-fp2}, satisfies the condition of the Borel conjecture in \cite{Bor79} and more generally in \cite[Conjecture 2]{Sol20}.

\begin{lem}\label{lem_borel}
    The morphism $\eta_\fp\colon H_\rho\longrightarrow G$ in \eqref{etafp} enjoys the following properties. 
    \begin{itemize}
        \item [(1)] The kernel of $\ud \eta_{\fp}$ is central.
        \item [(2)] The cokernel of $\eta_{\fp}$ is a commutative group defined over $F$.
    \end{itemize}
\end{lem}

\begin{proof}
    According to the proof of Proposition \ref{prp:es}, the morphism in \eqref{map0} and Corollary \ref{cor:Hrhosc}, the morphism $\eta_{\fp}$ is explicitly given by
    \[
    H_\rho\cong\BG_m\ltimes_{\lambda}G_0^\rsc\rightarrow G\colon (a,g)\mapsto \lambda(a)\bar{g}.
    \]
    Let $\Fg$ be the Lie algebra of $G$ and $\Fg_0:=[\Fg,\Fg]$, which is the Lie algebra of $G_0^{\rsc}$. Let $\Fg_a$ be the Lie algebra of $\BG_m$, then the morphism $\ud\eta$ is given by
    \[
    \ud\eta_{\fp}:\Fg_a\oplus\Fg_0\rightarrow\Fg: (A,X)\mapsto \ud\lambda(A)+X,
    \]
    where $\ud\lambda:\Fg_a\rightarrow\Fg$ is the differential of $\lambda:\BG_m\rightarrow G$. Hence we have the kernel of $\ud \eta$ given by 
    \[
    \{(A,-\ud\lambda(A))\mid A\in\Fg_a\}.
    \]
    According to the multiplication of $H_{\rho}$ given in Corollary \ref{cor:Hrhosc}, the Lie bracket in $\Fg_a\oplus\Fg_0$ is given as
    \[
    [(A_1,X_1),(A_2,X_2)]_{H_{\rho}}=(0,-[[\ud\lambda( A_2),X_1],X_2]),
    \]
    where the brackets on the right-hand side of the equality is the Lie bracket in $\Fg$. Then for any $A\in\Fg_a$, and $(A^{\prime},X^{\prime})\in\Fg_a\oplus\Fg_0$, we have
    \[
    [(A,-\ud\lambda(A)),(A^{\prime},X^{\prime})]_{H_{\rho}}=(0,[[\ud\lambda(A^{\prime}),\ud\lambda(A)],X^{\prime}])=0
    \]
    since $[\ud\lambda(A^{\prime}),\ud\lambda(A)]=\ud\lambda([A,A^{\prime}])=0$, where $[A,A^{\prime}]$ denotes the trivial Lie bracket in $\Fg_a$. This proves Part $(1)$.

    As for Part $(2)$, note that $G/\eta_{\fp}(H_{\rho})=G/\lambda(\BG_m)G_0$, and since $G/G_0$ is already abelian, we obtain that $G/\eta_{\fp}(H_{\rho})$ is also abelian and defined over $F$ since every group is defined over $F$.
We are done. \end{proof}

Assume that the general Borel conjecture holds for the morphism $\eta_{\fp}\colon H_{\rho}\to G$, which under the above Condition of the Lemma \ref{lem_borel} is proved for many cases in \cite[Theorem 3]{Sol20}.
We deduce that the $\sigma$ appearing in Theorem \ref{thm:localm} occurs as a direct summand in the direct sum decomposition of the pull-back of $\pi$ under the morphism $\eta_{\Fp}$, with expected multiplicity (\cite[Conjecture 2]{Sol20}. Hence Theorem \ref{thm:localm} can be improved 

\begin{thm}
    Assume that the local Langlands conjecture over $F$ holds for $(G,G^{\vee},\rho)$ and its associated 
    $\bkn$-triple $\Delta_\fp=(H_\rho,H_\rho^\vee,\rho_\fp)_{\Fd_\fp}=\Delta_\bkn(\Fd_\fp)$. Assume that the Borel conjecture holds for the 
    morphism $\eta_{\fp}\colon H_{\rho}\to G$ as in \eqref{etafp}. Then for any $\pi\in\Pi_F(G)$, there exists a $\sigma\in\Pi_F(H_{\rho})$, which occurs in the direct decomposition of the pull-back of $\pi$ 
    via the morphism $\eta_{\fp}$,     
    such that 
    \[
    L(s,\pi,\rho)=L(s,\sigma,\rho_{\fp}). 
    \]
\end{thm}

\bibliographystyle{alpha}
%    Insert the bibliography data here.
	\bibliography{references}
\end{document}